\documentclass[reqno,11pt, a4paper]{amsart}  %reqno keeps the equations numbering on the right hand of the page
\usepackage{article_ams_template}%Contains all self-defined macros and loaded packages

\begin{document}
%opening
\title[Simulation of multivariate W-F diffusions]{Exact simulation of coupled Wright-Fisher diffusions}      

\author[García-Pareja, Hult \and Koski]{Celia García-Pareja}
\address[Celia García-Pareja*]{Unit of Biostatistics, IMM, Karolinska Institutet}
\curraddr{Department of Mathematics, KTH Royal Institute of Technology}
\email[Corresponding author*]{celiagp@kth.se}
\author[]{Henrik Hult \and Timo Koski}
\address[Henrik Hult \and Timo Koski]{Department of Mathematics, KTH Royal Institute of Technology}
\email[]{hult@kth.se, tjtkoski@kth.se}

%,epistasis}

\renewcommand{\thefootnote}{\fnsymbol{footnote}} 
\footnotetext{\emph{MSC2010 subject classifications.} Primary 60J60, 65C30; secondary 60H35, 65C05.}     
\renewcommand{\thefootnote}{\arabic{footnote}}

\begin{abstract}
In this paper an exact rejection algorithm for simulating paths of the coupled Wright-Fisher diffusion 
is introduced. The coupled Wright-Fisher diffusion is a family of 
multidimensional Wright-Fisher diffusions that have drifts depending on each other through a coupling term and that find applications in the study of interacting genes' networks as those encountered in studies of antibiotic resistance.
Our algorithm uses independent neutral Wright-Fisher 
diffusions as candidate proposals, which can be sampled exactly by means of existing algorithms 
and are only needed at a finite number of points. Once a candidate is accepted, the remaining of the path 
can be recovered by sampling from a neutral multivariate Wright-Fisher bridge, for which we also provide 
an exact sampling strategy. The technique relies on a modification of the alternating series method and 
extends existing algorithms that are currently available for the one-dimensional case. Finally, the
algorithm's complexity is derived and its performance demonstrated in a simulation study.
\end{abstract}

\maketitle

\section{Introduction}
%\section{Introduction}
Sampling paths of a diffusion process remains a challenging problem in applied probability. The major 
bottleneck is that their finite dimensional distributions are seldom available in closed form, 
and one often needs resorting to time-discretized numerical approximations. These approximations, however, 
induce bias 
and approximation errors that are difficult to quantify. Moreover, reducing such errors requires refining the 
time grid, which, in turn, increases computational costs. In this context, exact simulation algorithms, which aim to recover samples from the true finite dimensional distributions of a diffusion, %even if they are unknown in closed form, 
have become increasingly popular.

The standard approach to exact simulation of diffusions is based on the family of exact rejection algorithms, which 
rely on an acceptance-rejection scheme that requires samples 
from a candidate diffusion only at a finite collection of time points in order to take a decision. 
The candidate needs to be such 
that is possible to simulate without approximation, and that allows for the construction of the
acceptance-rejection probability by means of the Girsanov's transformation of measures, see \cite{Karatzas1998}. 
Once a candidate is accepted, the algorithm returns a skeleton of the target path, 
and the remaining segments 
can be sampled at any other time instance by simulating from suitable diffusion bridges and with no further 
reference 
to the unknown target distribution. 
In their seminal paper, Beskos and Roberts \cite{Beskos2005} present an exact rejection algorithm for 
simulation 
of paths of a certain class of one-dimensional diffusions, which requires imposing some boundedness assumptions 
on the drift of 
the target diffusion and its derivative. Acknowledged as too restrictive, 
these assumptions are relaxed to one-sided bounds in a second publication, see \cite{Beskos2006}, and a further extension based on a layered Brownian bridge 
construction 
provides an exact simulation method 
for boundary 
crossing and hitting times, see \cite{Beskos2008}. Other extensions include algorithms that allow for 
simulation of killed diffusions and have applications to double barrier option pricing problems 
\cite{Casella2008}, or a localized exact algorithm that relaxes any boundedness assumptions by 
considering smaller pieces of the target path and can be used to simulate diffusions with boundaries 
\cite{Chen2013}. 
An interesting approach is also that presented in \cite{Pollock2016}, where the authors provide an exact 
rejection algorithm for jump diffusions. 

Downsides of the exact rejection algorithms presented above include that they impose somewhat strong 
assumptions on the drift and require the use of the Lamperti transformation, see \cite{Kloeden1992}, 
to obtain unit volatility coefficients, which hinders generalizability to the multivariate case. 
To overcome these issues, alternative techniques have been proposed, such as the one in \cite{Blanchet2017} 
that introduces an exact algorithm for simulation of multivariate diffusions based on tolerance 
enforced simulation and rough paths analysis. This algorithm overcomes the more restrictive assumptions 
required in \cite{Beskos2005} and \cite{Beskos2006} but has, admittedly, infinite expected running time. 

Another restrictive feature of exact rejection algorithms is that they rely heavily on the availability 
of suitable candidates (in all cases mentioned above, Brownian motion or slight modifications thereof). 
For diffusions with finite boundaries, for example, 
Brownian candidates can either differ too much from the target, thus providing low %too low 
acceptance probabilities, or be unsuitable to construct the acceptance-rejection probability itself. 
Rejection algorithms with candidates other than Brownian motion include \cite{Jenkins2013} that uses Bessel proposals to simulate a certain class of diffusions 
with a finite entrance boundary.

In this context, the recent work in \cite{Jenkins2017} extends the class of diffusions for which exact 
rejection 
simulation is possible. The authors propose
a simulation technique to recover samples from neutral Wright-Fisher diffusions that, in turn, are used 
as candidates 
in an exact rejection algorithm for simulating a wider target family of one-dimensional Wright-Fisher diffusions. This class of diffusions, as well as its multivariate counterparts, are extensively used in population genetics, where proliferation of exact simulation algorithms can foster the use of suitable inferential techniques such as approximate Bayesian computation, see \cite{Tavare1997}, \cite{Beaumont2002}, \cite{Fearnhead2012}.

Along these lines, a main contribution of this paper is to present an exact rejection algorithm for coupled Wright-Fisher diffusions, with candidates %for our algorithm are 
built from samples of independent neutral Wright-Fisher diffusions that can be recovered using the techniques 
presented in \cite{Jenkins2017}.
The coupled Wright-Fisher diffusion \cite{Aurell2019a}, %\cite{Aurell2019},
is a family of multivariate Wright-Fisher diffusions that models how different allele types (genetic traits) co-evolve across different loci (different locations along the genome), over generations. This type of diffusion model is used to analyse networks of loci in recombining populations of bacteria (e.g.\ Streptococcus Pneumoniae) under strong selective pressure, when the linkage disequilibrium is low across the genome, see \cite{Ekeberg2013, Skwark2017}. It incorporates parent dependent mutation, interlocus selection and free recombination. In contrast, the model is unsuitable in populations of bacteria  where the amount of homologous recombination is low, which makes it difficult to separate couplings arising from recombination from those arising from selection (e.g. Streptococcus Pyogenes). Moreover, diffusion approximations have been deemed poor in some scenarios, e.g. for low mutation rates where the stationary density is ill-defined at the boundary, see \cite{Hoessjer2016}. 

The coupled Wright-Fisher  
is based on quasi-linkage equilibrium where the fitness coefficients are inspired by a Potts model, see \cite{Aurell2019a,Shraiman2011},  
and generalize the classical additive fitness under weak selection, see e.g. \cite[Ch. II]{Burger2000}, to the multi-locus case.  With two loci and without the first order selection terms, the coupled Wright-Fisher diffusion corresponds to a haploid version of the model with weak selection, loose linkage in \cite{ethier1989}, see also \cite{Fearnhead2006} for a different multi-locus extension. The coalescent model associated to the coupled Wright-Fisher diffusion, describing the ancestral history of a sample of individuals, is derived in \cite{Favero2020} using Markov duality.

To complete the proposed exact rejection algorithm, a further contribution of this paper deals with simulation of multidimensional 
Wright-Fisher bridges, for which we present an exact simulation technique. These bridges allow sampling further points of the path once a skeleton of the coupled Wright-Fisher 
diffusion has been accepted. Our sampling approach can therefore be viewed as a generalization of 
that presented in \cite{Jenkins2017} for the one-dimensional Wright-Fisher diffusions to the multivariate case.

The rest of the paper is structured as follows. In Section \ref{sec:background}, main properties and 
structure of the family of coupled Wright-Fisher diffusions are briefly presented jointly with a formal 
overview on exact rejection algorithms. In Section \ref{sec:candidate} we recall and present some revised 
algorithms for exact simulation of one and multidimensional neutral Wright-Fisher diffusions, i.e., 
those needed  for sampling our candidate processes, leading up to the proposed exact rejection algorithm 
for coupled Wright-Fisher diffusions (Section \ref{sec:exact_alg_coupled}). Section \ref{sec:num_exp} 
includes performance results illustrated through several simulation scenarios and in Section \ref{sec:bridge} 
the technique for simulating exactly from a multidimensional Wright-Fisher bridge is provided, which completes 
the sampling scheme. 
Finally, Section \ref{sec:proofs} contains mathematical proofs.

\section{Background}\label{sec:background}
%\section{Background}
This section provides the necessary insights on the structure and main properties of the coupled Wright-Fisher family of diffusions and fixes some notation, as well as provides a brief overview of exact rejection algorithms for diffusions, which constitute the basis of our work.
\subsection{Coupled Wright-Fisher diffusions}
The family of Wright-Fisher models, and more specifically their diffusion approximations, have been 
widely used in population genetics, see, for instance, \cite{Kimura1964} and \cite{Griffiths1980}%,\textcolor{red}{ADD more recent}
. In its simplest form, the Wright-Fisher 
model describes the evolution of the frequency of two allele types in a single locus 
that have the same fitness, and whose configuration at each new generation of individuals is 
chosen uniformly and with replacement from that of the current generation in an haploid population of constant size.  
Extensions of the  
model include considering more than two allele types that might be located at different loci, and can incorporate 
other evolutionary forces such as mutation, selection and recombination. A comprehensive overview on the family of 
Wright-Fisher models can be found, for example, in \cite{Crow1970} or \cite{Ewens2004}.

With the proliferation of genome-wide association studies, questions arise about how genetic variants associated to numerous 
diseases co-evolve or interact over time. Moreover, the increasing availability of allele frequency time series data is fostering 
the study of evolutionary forces such as mutation or selection, see \cite{Steinrucken2014}, 
\cite{Skwark2017}, \cite{Tataru2017}, \cite{Nene2018}. Within this framework, the recently proposed coupled Wright-Fisher model \cite{Aurell2019a} %\cite{Aurell2019} 
tracks the evolution of frequencies of allele types located at different loci, 
and, besides locus-wise mutations, describes possible selective pairwise interactions between allele type frequencies 
across different loci in an haploid population. Its diffusion approximation can be derived as the weak limit of a sequence of discrete Wright-Fisher models characterized by the assumption that the evolution of the population at one locus is conditionally independent of the other loci given the state of the previous generation. 
The coupled Wright-Fisher diffusion can be expressed as a 
system of stochastic differential equations of the form 
\begin{equation}\label{eq:coupled_WF_general}
dX_t=[\alpha(X_t)+G(X_t)]dt+ D^{\frac{1}{2}}(X_t)dB_t, \ X_0=x_0,\ t\in[0,T],
\end{equation}
where $X_t$ is a vector of frequencies of allele types, $\alpha$ governs their mutations  and $G$ contains the single and pairwise selective locus interactions.
   
Let $L$ denote the total number of loci and $d_i\geq 2$ the number of different allele types in each of them. For $n:=\sum_{i=1}^L [d_i-1]$, let us index each element of an $n$-dimensional vector $x$ by its referring to a 
specific locus $i\in\{1,\ldots , L\}$ and allele type $j\in\{1,\ldots, d_i-1\}$, so that $x=\{x^i\}_{i=1}^L$ 
where each $x^i=\{x^{ij}\}_{j=1}^{d_i-1}$. If $x\in\mathbf{R}^n$ refers to the vector of allele type frequencies $X_t$, 
the elements of the drift $\alpha(x)\in\mathbf{R}^n$ take the form 
\begin{equation}\label{eq:drift_coupled_WF_general}
\alpha^{ij}(x^{ij})=\frac{1}{2}\left(\theta^{i}_{j}-|\theta|x^{ij}\right),
%\ j\in\{1,\ldots, m\}, i\in\{1,\ldots , L\},
\end{equation}
where $|\theta|=\sum_{k=1}^{d_i}\theta^{i}_{k}$ and $\theta^{i}_{k}>0$ denote the parent-independent mutation rates to allele type $k\in\{1,\ldots, d_i\}$ at locus $i$, 
so that mutations occur at each locus separately. Wright-Fisher diffusions with drift $\alpha(x)$ correspond to the reversible neutral mutations allele model.  %(see, for example, )\textcolor{red}{REF?}. 
The coupled Wright-Fisher model also admits  
parent-dependent mutations, but we will not consider them here.

The coupling term $G(x)\in\mathbf{R}^n$ has general form
\begin{equation}
G(x)=D(x)\nabla_{x} (\overline{V}\circ f)(x),
\end{equation}
where the square of the diffusion matrix $D(x)=\text{diag}(D^i(x^i))\in\mathbf{R}^{n\times n}$ is an 
$L$-blocks diagonal matrix with entries 
$$
D^i_{jk}=
\begin{cases}
  x^{ij}(1-x^{ij}), j=k, \\
  -x^{ij}x^{ik}, j\neq k,
\end{cases} j,k \in\{1,\ldots, d_i-1\}, i\in\{1,\ldots , L\},
$$
$\nabla_{x}$ is the gradient operator w.r.t.~each component of $x$, $f$ transforms $x$ in the augmented 
$(n+L)$-dimensional vector 
$\overline{x}$ that reflects the dependency between allele frequencies at locus $i$, i.e., 
$$
\overline{x}^{ik}:=f^{ik}(x)=
\bigg\{
\begin{array}{cl}
 x^{ij},& j=k\in\{1,\ldots, d_i-1\},\\
 1-\sum_{j=1}^{d_i-1}x^{ij},& k=d_i,\\
\end{array}$$
and $\overline{V}(\overline{x})\in \mathbf{R}$,
\begin{equation}
\overline{V}(\overline{x})=(\overline{x})^Ts+\frac{1}{2}(\overline{x})^TH\overline{x},
\end{equation}
where $s\in\mathbf{R}^{n+L}$ is a within locus selection parameters vector and 
$H\in\mathbf{R}^{(n+L)\times(n+L)}$ is a symmetric across-loci pairwise 
interactions matrix. The matrix $H$ is in fact built by $L$ blocks of zeros of size $d_i\times d_i$ in the main diagonal 
(denoted $0^{ii}$), and off-diagonal blocks of the form $H^{il}=(H^{li})^T\in\mathbf{R}^{d_i\times d_l}, i\neq l, i,l\in\{1,\ldots, L\}$,
$$
H=
\begin{pmatrix}
 0^{11}&H^{12}&\ldots&\ldots&\ldots&\ldots&H^{1L}\\
 \vdots&\ddots&\vdots&\vdots&\vdots &\vdots &\vdots\\
 H^{i1}&\ldots&H^{i(i-1)}&0^{ii}&H^{i(i+1)}&\ldots&H^{iL} \\
 \vdots&\ddots&\vdots&\vdots&\ddots &\vdots &\vdots\\
 H^{L1}&\ldots&\ldots&\ldots&\ldots&\ldots&0^{LL} \\
\end{pmatrix},
$$
so that interactions of each locus with itself are not permitted. Note that if one removes the coupling term $G$, 
(\ref{eq:coupled_WF_general}) simply becomes the usual multidimensional neutral mutations Wright-Fisher diffusion, 
where $X_t$ describes the evolution of allele frequencies that evolve independently at each locus. The explicit form of $V(x):=\nabla_{x} (\overline{V}\circ f)(x)$ in terms of $s$ and $H$ is specified in the 
following proposition, whose proof can be found in Section \ref{sec:proofs}.

\begin{prop}\label{prop:gradient}
Let $V(x)\in\mathbf{R}^n$ be the $n$-dimensional vector such that $V(x):=\nabla_{x} (\overline{V}\circ f)(x)$.
Then $\forall i \in\{1,\ldots, L\}, j\in\{1,\ldots, d_i-1\}$
 \begin{equation*}
V^{ij}(x)=
K_s^{ij}+\sum_{\underset{l\neq i}{l=1}}^{L}\bigg(K_l^{ij}+ \sum_{k=1}^{d_l-1} K_{lk}^{ij} x^{lk}\bigg), 
\end{equation*}
with $$K_s^{ij}:=s^{ij}- s^{id_i}, \quad K_l^{ij}:=h^{il}_{jd_l}-h^{il}_{d_id_l} \text{ and } 
K_{lk}^{ij}:=h^{il}_{jk}-h^{il}_{jd_l}-h^{il}_{d_ik}+h^{il}_{d_id_l},$$ where 
$h^{il}_{jk}$ denotes the $j$th-row-$k$th-column entry of the block $H^{il}$ of $H$.
\end{prop}

Following Kimura's formulation \cite{Kimur1955}, the explicit stationary density of (\ref{eq:coupled_WF_general}) can 
also be obtained by solving the corresponding Fokker-Planck equation, see \cite{Aurell2019a}, %\cite{Aurell2019}, 
whose solution takes the form
\begin{equation}\label{eq:stationary}
 P(x)=\dfrac{1}{Z}\pi(x)e^{2(\overline{V}\circ f)(x)},
\end{equation}
where
$$
\pi(x)=\prod_{i=1}^L \pi^i(x^i)=
\prod_{i=1}^L \left( (1-\sum_{j=1}^{d_i-1}x^{ij})^{\theta^i_{d_i}-1}\prod_{j=1}^{d_i-1}(x^{ij})^{\theta^i_j-1}\right)
$$
and $Z$ is the normalizing constant
$$
Z=\int_{x\in\mathcal{X}} \pi(x)e^{2(\overline{V}\circ f)(x)}dx.
$$
Here $\mathcal{X}=\{x\in\mathbf{R}^n\mid x^{ij}\geq 0, \sum_{j=1}^{d_i-1}x^{ij}\leq 1\}$.

The representation \eqref{eq:stationary} resembles the stationary density of the haploid version of the model studied by Fearnhead, see \cite{Fearnhead2006} Theorem 2, which in the two loci case and with vanishing within-locus selection, agrees with the coupled Wright-Fisher diffusion. %In the case of two loci Fearnhead's model agrees with the  with vanishing within locus selection, but in the general multi-locus case they are different. 

\subsection{Overview of exact rejection algorithms for simulation of diffusions}
This subsection provides an overview on the exact rejection algorithm presented in \cite{Beskos2005}, and presents the same sampling scheme followed for simulating from the coupled Wright-Fisher diffusion, detailed further in Algorithm \ref{alg:CWFdiff}. Let 
\begin{equation}
 dX_t= \mu(X_t) dt+dB_t,\ X_0=x_0,\ t\in [0,T],\label{eq:1dim_diffusion}
\end{equation}
where $\mu(\cdot)$ is such that (\ref{eq:1dim_diffusion}) admits a unique 
weak solution $(X_t)_{t\in[0,T]}$.
Let $\mathbb{Q}_{x_0}$ be the law of the process $(X_t)_{t\in[0,T]}$ and 
let $\mathbb{P}_{x_0}$ denote the law of a Brownian 
motion $(B_t)_{t\in[0,T]}$ starting at $B_0=x_0$. By means of the 
Girsanov's transformation of measures one can write the Radon-Nykod\'{y}m derivative of 
$\mathbb{Q}_{x_0}$ w.r.t. $\mathbb{P}_{x_0}$
\begin{equation}\label{eq:Girsanov_1dim_diff}
 \dfrac{d\mathbb{Q}_{x_0}}{d\mathbb{P}_{x_0}}=\exp\left\{\int_0^T \mu(B_t) dX_t-\dfrac{1}{2}\int_0^T \mu^2(B_t)dt\right\}.
\end{equation}
Assuming that $\mu(\cdot)$ is differentiable everywhere and using It\^{o}'s lemma, (\ref{eq:Girsanov_1dim_diff}) can be rewritten as
\begin{equation}\label{eq:itolemma_1dim_diff}
 \dfrac{d\mathbb{Q}_{x_0}}{d\mathbb{P}_{x_0}}=\exp\left\{\tilde{A}(B_T)-\tilde{A}(x_0)\right\}\exp\left\{-\dfrac{1}{2}\int_0^T (\mu^2(B_t)+\mu'(B_t)) dt\right\},
\end{equation}
where  $\tilde{A}(x):=\int_0^x\mu(u)du$. Imposing the further conditions of $\tilde{A}(x)$ to be bounded above by a constant $K^A$, and $(\mu^2+\mu')/2$ 
to be bounded between constants $K^-$ and $K^+$, 
\begin{equation}\label{eq:rejprob_1dim_diff}
 \dfrac{d\mathbb{Q}_{x_0}}{d\mathbb{P}_{x_0}}\propto\exp\left\{\tilde{A}(B_T)-K^A\right\}
 \exp\left\{-\int_0^T (\tilde{\phi}(B_t) -K^- )dt\right\},%\leq 1,
\end{equation}
with $\tilde{A}(x)\leq K^A$ and $K^-\leq \tilde{\phi}(x):=\dfrac{1}{2}[\mu^2(x)+\mu'(x)]\leq K^+$, is a suitable acceptance-rejection probability. 

Note that, a priori, an exact evaluation of the integral in (\ref{eq:rejprob_1dim_diff}) is not possible without any approximation error because it would require to store infinitely many points of the sample candidate path $B=(B_t)_{t\in[0,T]}$. However, the key point of the exact algorithms proposed in \cite{Beskos2005} and publications therein, lies on the fact that an event occurring with probability (\ref{eq:rejprob_1dim_diff}) can be 
evaluated sampling $B$ only at a finite number of time points. This follows because the last term in (\ref{eq:rejprob_1dim_diff}) can be interpreted as the probability of the event $\omega_{\tilde{\phi}}$
that no points from an homogeneous spatial Poisson process $\Phi=\{(t_j,\psi_j): j=1,\ldots,J\}$ with unit intensity on $[0,T]\times[0,K^+-K^-]$ lie below the graph of $t\mapsto\tilde{\phi}(B_t)$. A formal statement and proof of this observation can be found in Theorem 1 of \cite{Beskos2006}. 

Therefore, the exact sampling procedure (detailed in Algorithm \ref{alg:Beskos}) starts by drawing a sample from $\Phi$ that will determine the time points at which the candidate $B$ will be drawn, and then provides a skeleton of $(X_t)_{t\in[0,T]}$ at such time points upon acceptance of the sampled candidate (that is, if all the evaluated $\tilde{\phi}(B_t)$ lie above the sampled Poisson points). Note that the last point on the candidate path, $B_T$, serves to evaluate an event that occurs with probability $\exp\{\tilde{A}(B_T)-K^A\}$ and that is independent of $\omega_{\tilde{\phi}}$. In an original version of the algorithm, $B_T$ is sampled from a slightly modified distribution; a biased Brownian motion that serves as a valid candidate and improves the algorithm's efficiency. Such option is not needed for our purposes and is therefore not fully described here. In brief, this modification permits relaxing the boundedness condition on $\tilde{A}(x)$, but the equivalent function in our proposed exact algorithm is already bounded.
\begin{algorithm}
\caption{Exact algorithm for simulating skeletons of paths $(X_t)_{t\in[0,T]}$ of a diffusion process with law 
$\mathbb{Q}_{x_0}$}\label{alg:Beskos}
\algsetup{linenodelimiter=}
\begin{algorithmic}[1] % enter the algorithmic environment
    %\REPEAT
    \STATE Simulate $\mathbf{\Phi}$, a Poisson process on $[0,T]\times[0, K^+-K^-]$.
    \STATE Simulate $U\sim \mbox{Uniform}(0,1)$
    \STATE Given $\Phi=\{(t_j,\psi_j): j=1,\ldots,J\}$, simulate $B\sim \mathbb{P}_{x_0}$ at times
    $\{t_1,\ldots,t_J\}$ and at time $T$.
    \IF{$\tilde{\phi}(B_{t_j})-K^-\leq \psi_j,\ \forall j$ and $U\leq \exp\{\tilde{A}(B_T)-K^A\}$}
        \RETURN $\{(t_j,B_{t_j}),\ \forall j\}\cup \{(T,B_T)\}$ 
    \ELSE \STATE Go back to Step 1. 
    \ENDIF
    %\UNTIL{false}
\end{algorithmic}
\end{algorithm}

\section{Simulation of the candidate processes}\label{sec:candidate}
%\section{Simulation of the candidate processes}
This section is devoted to describing existing simulation strategies for the candidate processes in the exact rejection algorithm that will be presented later in Section \ref{sec:exact_alg_coupled}. Suitable candidate processes in our setting will be $L$ independent $(d_i-1)$-dimensional neutral Wright-Fisher diffusions $(X_t)_{t\in [0,T]}$, each one unique weak solution of 
\begin{equation}\label{eq:candidate_Mi}
dX_t= \alpha(X_t) dt+D^{\frac{1}{2}}dB_t,\quad X_0=x_0,\ t\in [0,T],
\end{equation}
with $\alpha(X_t)$ a $(d_i-1)$-dimensional vector with $\alpha^{ij}(x)=\frac{1}{2}(\theta^i_j-|\theta| x^{ij})$.

\subsection{Transition density function expansions}
Exact simulation of each neutral Wright-Fisher diffusion is possible by exploiting available transition density's eigenfunction expansions that allow a probabilistic representation, see, for example, \cite{Griffiths2010}. 

For a fixed locus $i\in\{1,\ldots, L\}$, let $x=(x_1,\ldots, x_{d-1})$ be a vector of initial frequencies and $\theta+l$ a 
$d$-dimensional vector with entries $\theta_j+l_j, \ j\in\{1,\ldots, d\}$. Then, the probabilistic representation of the transition density function of $(X_t)_{t\in [0,T]}$ in (\ref{eq:candidate_Mi}) is given by %, $g(x,\cdot; t)$, 
\begin{equation}\label{eq:mult_trans_den}
g(x,y;t)=\sum_{m=0}^{\infty} q_m^{\theta}(t) \sum_{\underset{|l|=m}{l}} \mathcal{M}_{m,x}(l)\mathcal{D}_{\theta+l}(y) ,
\end{equation}
where $q_m^{\theta}(t)$ are transition functions of a pure death process $A_{\infty}^{\theta}(t)$ with an 
entrance boundary at $\infty$, $\mathcal{M}_{m,x}(\cdot)$ is the probability mass function (PMF) of a multinomial random variable, and $\mathcal{D}_{\theta+l}(\cdot)$ the probability density function (PDF) 
of a Dirichlet random variable, that is,
$$
 \mathcal{M}_{m,x}(l)=\frac{m!}{\prod_{j=1}^{d}l_j!}(1-\sum_{j=1}^{d-1}x_j)^{l_{d}}
 \prod_{j=1}^{d-1}x_j^{l_j},
 $$
 and
 $$
 \mathcal{D}_{\theta+l}(y)=\frac{\Gamma( |\theta+l|)}{\prod_{j=1}^{d}\Gamma(\theta_j+l_j)}
 (1-\sum_{j=1}^{d-1}y_j)^{\theta_{d}+l_{d}-1}\prod_{j=1}^{d-1}y_j^{\theta_j+l_j-1},
 $$
 with $l_d=m-\sum_{j=1}^{d-1} l_j$. A more detailed description of the process $A_{\infty}^{\theta}(t)$ as well as an exact sampling technique are provided in detail later on. 
 
In case of one-dimensional ($d=2$) Wright-Fisher diffusions, the multivariate components on the mixture in (\ref{eq:mult_trans_den}) reduce to their one-dimensional counterparts, i.e., a binomial and a beta random variables respectively \cite{Griffiths2010}.
% $\mathcal{D}_{\theta_1+l,\theta_2+m-l}(\cdot)$ the PDF of a beta random 
% their transition density function is
% \begin{equation}\label{eq:one_dimtrans}
% \tilde{g}(x,y;t)=\sum_{m=0}^{\infty} q_m^{\theta}(t) \sum_{l=0}^{m} \mathcal{B}_{m,x}(l)
% \mathcal{D}_{\theta_1+l,\theta_2+m-l}(y), 
% \end{equation}
% where $\mathcal{B}_{m,x}(\cdot)$ is the PMF of a binomial random variable and 
% $\mathcal{D}_{\theta_1+l,\theta_2+m-l}(\cdot)$ the PDF of a beta random 
% variable, see \cite{Griffiths2010}.

\begin{algorithm}
\caption{Exact simulation of samples from $g(x,\cdot\ ;t)$, transition density of the $(d-1)$-dimensional 
neutral Wright-Fisher diffusion with recursive mutation}\label{alg:multidim_neutralWF}
\algsetup{linenodelimiter=}
\begin{algorithmic}[1] % enter the algorithmic environment
    \STATE Simulate $M\sim A_{\infty}^{\theta}(t)$.
    \STATE Given $A_{\infty}^{\theta}(t)=m$, simulate $L\sim \text{Multinomial}(m,x)$.
    \STATE Given $L=(l_1,\ldots,l_{d-1})$, simulate $Y\sim\text{Dirichlet}(\theta+l)$
    \RETURN $Y=(y_1,\ldots, y_{d-1})$.
  \end{algorithmic}
\end{algorithm}

 A sampling strategy for $g(x,\cdot;t)$ is summarized in Algorithm \ref{alg:multidim_neutralWF}, 
 see  \cite{Griffiths1983} or \cite{Jenkins2017} for an analogous version in the one-dimensional case, 
 the latter also including a modification for the infinite-dimensional case, that is, 
 for Fleming-Viot diffusions. Once expressed in probabilistic terms and given the simplicity of 
Algorithm \ref{alg:multidim_neutralWF}, recovering samples from $g(x,\cdot;t)$ seems straightforward. 
However, sampling exactly from $q_m^{\theta}(t)$ poses some difficulty because it is only known in infinite 
series form. Previous approaches for simulating from approximated versions of $q_m^{\theta}(t)$ can be found 
in \cite{Griffiths1983}, \cite{Griffiths1984}, \cite{Griffiths2006} or \cite{Jewett2014}. 
In the next section, we review the exact simulation procedure presented in \cite{Jenkins2017}, which is 
the one used here.

\subsection{Exact simulation of the ancestral process $A_{\infty}^{\theta}$}\label{subsec:ancestral}
We describe here the sampling procedure for recovering exact samples of $q_m^{\theta}(t)$, transition functions of the aforementioned death process $A_{\infty}^{\theta}$. In more detail, let $\{A_n^{\theta}(t): t\geq 0\}$ be a pure death process on $\mathbb{N}$ such that $A_n^{\theta}(0)=n$ almost surely and with its only possible transition $m\to m-1$ occurring at rate $m(m+|\theta|-1)/2$ for each $m=1,\ldots,n$, that is, it represents the number of non-mutant lineages that coalesce backwards in time in the coalescent process with mutation. Then, let $q_m^{\theta}(t)=\lim_{n \to \infty}\Pr(A_n^{\theta}(t)=m)$.

For a more thorough interpretation of the transition density $g(x,\cdot, t)$ and its one-dimensional counterpart, it is worth noting that the expansion in (\ref{eq:mult_trans_den}) is derived via a duality principle for Markov processes \cite{Ethier2009}, that is, from the moment dual process of the Wright-Fisher diffusion, which is also a pure death process representing lineages backwards in time, see for example \cite{Etheridge2009} or \cite{Griffiths2010} for a complete derivation and details.

An expression for $q_m^{\theta}(t)$ starting from the entrance boundary at infinity is, see \cite{Griffiths1980},
\begin{equation}\label{eq:qm}
 q_m^{\theta}(t)=\sum^\infty_{i=0}(-1)^{i}b_{m+i}^{(t,\theta)}(m),
% q_m^{\theta}(t)=\sum^\infty_{k=m}(-1)^{(k-m)}b_k^{(t,\theta)}(m),
\end{equation}
where
\begin{equation}\label{eq:b_coefs}
b_{m+i}^{(t,\theta)}(m)=\frac{(|\theta|+2(m+i)-1)}{m!i!}
 \frac{\Gamma(|\theta|+2m+i-1)}{\Gamma(|\theta|+m)}e^{(m+i)(m+i+|\theta|-1)t/2}.
% b_k^{(t,\theta)}(m)=\frac{(\theta+2k-1)}{m!(k-m)!}
% \frac{\Gamma(\theta+m+k-1)}{\Gamma(\theta+m)}e^{k(k+\theta-1)t/2}.
\end{equation}
As shown in \cite{Jenkins2017}, samples from $q_m^{\theta}(t)$ can be recovered exactly by means of a variant of the alternating series method, described in \cite{Devroye2006}, Chapter 4. 
In brief, the alternating series method would require the sequence of coefficients $b_{m+i}^{(t,\theta)}(m)$ to be decreasing in $i$ for each $m$, condition that is not always met here. Nonetheless, one can exploit that  there exists a finite $C_m^{(t,\theta)}$ such that for all $m$ and 
$i\geq C_m^{(t,\theta)}$ the sequence of coefficients $b^{(t,\theta)}_{m+i}(m)$ decreases monotonically as $i$ tends to $\infty$. 
More explicitly, there exists
\begin{equation}\label{eq:C_constant}
C_m^{(t,\theta)}=\inf\left\{i\geq 0 : (b^{(t,\theta)}_{m+i+1}(m)/b^{(t,\theta)}_{m+i}(m))<1 \right\}<\infty.
\end{equation}
Then, once $C_m^{(t,\theta)}$ is available, the remaining of the sequence of coefficients is ensured to be decreasing and the alternating series method can be applied. The following proposition summarizes the main properties of the bound $C_m^{(t,\theta)}$.
\begin{prop}\label{prop:bound_m}[Proposition 1 in \cite{Jenkins2017}]
 Let $b_{m+i}^{(t,\theta)}(m)$ be the coefficients defined in (\ref{eq:b_coefs}) and $C_m^{(t,\theta)}$ be as 
 in (\ref{eq:C_constant}). Then
 \begin{enumerate}
  \item[i)] $C_m^{(t,\theta)}<\infty$ for all $m$.
  \item[ii)] $b_{m+i}^{(t,\theta)}(m)\downarrow 0$ as $i\to\infty$ for all %$m$ and 
  $i\geq C_m^{(t,\theta)}$.
  \item[iii)] $C_m^{(t,\theta)}=0$ for all $m>D_{\epsilon}^{(t,\theta)}$, where
  \begin{equation}\label{eq:D_constant}
    D_{\epsilon}^{(t,\theta)}=%checḱ what to do with the epsilon
    \inf\left\{u\geq \left(\frac{1}{t}-\frac{|\theta|+1}{2}\right)\vee 0 :
   (|\theta|+2u-1)e^{u(u+|\theta|-1)t/2}<1-\epsilon \right\},
   \end{equation}
for $\epsilon\in[0,1)$.
 \end{enumerate}
\end{prop}
Property iii) in Proposition \ref{prop:bound_m} will be of interest later when proposing the exact sampling algorithm for $(d-1)$-dimensional Wright-Fisher bridges (Section \ref{sec:bridge}), where an explicit bound on $m$ is needed.

One can then recover exact samples from $q_m^{\theta}(t)$ because the terms $b^{(t,\theta)}_{m+i}(m)$ become monotonically smaller with increasing $i$, and for each $m$ there exist $k_m$, elements of $\tilde{k}\in\mathbf{R}^{M+1}$ (i.e., $\tilde{k}=\{k_m\}_{m=0}^M$) such that
$$
S^-_{\tilde{k}}(M):=\sum_{m=0}^M\sum^{2k_m+1}_{i=0}(-1)^{i}b_{m+i}^{(t,\theta)}(m)
\leq \sum_{m=0}^M q_m^{\theta}(t)
\leq \sum_{m=0}^M\sum^{2k_m}_{i=0}(-1)^{i}b_{m+i}^{(t,\theta)}(m)=:S^+_{\tilde{k}}(M).
$$
Because 
$$
\limsup_{\tilde{k} \to (\infty, \ldots, \infty)}S^-_{\tilde{k}}(M)= \Pr(A_{\infty}^{\theta}(t)\leq M)
\text{ and }
\liminf_{\tilde{k} \to (\infty, \ldots, \infty)}S^+_{\tilde{k}}(M)= \Pr(A_{\infty}^{\theta}(t)\leq M),
$$

%when $k_m\to \infty$ for each $m$, $S^-_{\tilde{k}}(M) \uparrow \Pr(A_{\infty}^{\theta}(t)\leq M)$ and 
% $S^+_{\tilde{k}}(M)\downarrow\Pr(A_{\infty}^{\theta}(t)\leq M)$, 
and both $S^-_{\tilde{k}}(M)$ and $S^+_{\tilde{k}}(M)$ can be computed from finitely many terms, given 
$U\sim \text{Uniform}(0,1)$ we can find $\tilde{k}^0\in\mathbf{R}^{M+1}$ with elements $k^0_m$ such that
\begin{equation*}
k^0_m=\inf\left\{ k_m\in\mathbb{N}: S^-_{\tilde{k}}(M) > U \text{ or } 
S^+_{\tilde{k}}(M) < U\right\},
\end{equation*}
for each $m\in\{0,\ldots, M\}$.

Now, if $\tilde{k}^0$ is such that $S^-_{\tilde{k}^0}(M) > U$, standard inverse sampling provides 
\begin{equation}\label{eq:k_indexs}
\inf\left\{ M\in\mathbb{N}: \sum_{m=0}^M q_m^{\theta}(t) \geq S^-_{\tilde{k}^0}(M)> U\right\},
\end{equation}
with $M$ exactly distributed following $q_m^{\theta}(t)$. The sampling strategy will consist in exploring the summands in $S^-_{\tilde{k}}(M)$ and $S^+_{\tilde{k}}(M)$ through their respective indexes $m$ and $k_m$, until, for a given realization of $U$, condition (\ref{eq:k_indexs}) is satisfied. 

A complete simulation procedure is presented in Algorithm \ref{alg:deathprocess}, where several improvements mentioned in \cite{Jenkins2017} have been incorporated. Most notably, the variable $M$ is initialized at the nearest integer around the mean of a certain distribution (not necessarily at $0$) that serves as an estimate of the mode $\hat{q}_{\text{mod}}$ of $q_m^{\theta}(t)$, a modification that decreases
computation times substantially. Such initialization originates from an asymptotic approximation of the transition functions $q_m^{\theta}(t)$, that first appeared in \cite{Griffiths1984}, which states that as $t\to 0$, $A_{\infty}^{\theta}(t)$ converges to a normal distribution, that is, 
\begin{equation}\label{eq:asymp_aprox}
\frac{A_{\infty}^{\theta}(t)-\mu^{(t, \theta)}}{\sigma^{(t, \theta)}}\stackrel{\mathcal{D}}{\longrightarrow} \mathcal{N}(0,1)\text{ as } t\to 0,
\end{equation}
where $\mu^{(t, \theta)}=2\eta/t$,
$(\sigma^{(t, \theta)})^2=\begin{cases} \frac{2\eta}{t\beta^2}(\eta+\beta)^2
\left(1+\frac{\eta}{\eta+\beta}-2\eta\right), \quad \beta\neq 0\\ 
\frac{2}{3t}, \quad \beta= 0\end{cases}$, 
\\with $\eta=\beta/e^{\beta}-1$ for $\beta\neq 0$ or $\eta=1$ otherwise, $\beta=\frac{1}{2}(|\theta|-1)t$, and where $\stackrel{\mathcal{D}}{\longrightarrow}$ denotes convergence in distribution, see Theorem 1 in \cite{Jenkins2017}.
\begin{algorithm}
\caption{Exact simulation of samples from $q_m^{\theta}(t)$, transition functions of the ancestral process $A_{\infty}^{\theta}$}\label{alg:deathprocess}
\algsetup{linenodelimiter=}
\begin{algorithmic}[1] % enter the algorithmic environment
    \STATE Set $M \gets \hat{q}_{\text{mod}}, \tilde{k}\gets(0,\ldots,0), j\gets1$
    \STATE Simulate $U\sim \mbox{Uniform}(0,1)$
    \REPEAT
    \FORALL{$m\in\{0,\ldots, M\}$} 
      \STATE Set $k_m\gets \lceil C_m^{(t,\theta)}/2 \rceil$ 
    \ENDFOR
    \WHILE{$S^-_{\tilde{k}}(M)<U<S^+_{\tilde{k}}(M)$} 
      \STATE Set $\tilde{k}\gets \tilde{k}+(1,\ldots,1)$
    \ENDWHILE
    \IF{$S^-_{\tilde{k}}(M)>U$}
      \RETURN $M$
    \ELSIF{$S^+_{\tilde{k}}(M)<U$} 
      \STATE Set $M\gets \hat{q}_{\text{mod}}+(-1)^{j}\lceil\frac{j}{2}\rceil$
      \IF{$j$ odd}
      \STATE $\tilde{k}\gets (k_0,\ldots, k_M)$ 
     \ELSIF{$j$ even}
        \STATE $\tilde{k}\gets (\tilde{k}, 0, \ldots, 0)$
	\ENDIF
	\ENDIF
     \STATE Set $j\gets j+1$
    \UNTIL false
  \end{algorithmic}
\end{algorithm}

Note that this initialization is possible because exploring the summands 
in $S^-_{\tilde{k}}(M)$ and $S^+_{\tilde{k}}(M)$ does not require to follow any specific order. 
It is also worth mentioning that, when $M$ is initalized at $0$, the vector $\tilde{k}$ is updated 
increasingly, i.e., a new element of the vector is added at every new iteration where $M$ is increased 
one unit. In Algorithm \ref{alg:deathprocess}, however, $M$ is updated telescopically, that is, at each 
new iteration $j$, $M$ moves farther from $\hat{q}_{\text{mod}}$ by one unit alternatingly above or below. This in turn, entails updating the 
corresponding $M+1$ elements of $\tilde{k}$ accordingly, i.e., the number of elements might either increase or 
decrease at each iteration. For precise results on the complexity of Algorithm \ref{alg:deathprocess} 
and simulation performance we refer the reader to \cite{Jenkins2017}.

At this stage, Algorithm \ref{alg:deathprocess} can be used in step 1 of Algorithm \ref{alg:multidim_neutralWF} and an exact sampling procedure for $(d-1)$-dimensional neutral Wright-Fisher diffusions is completed. 

\section{Exact rejection algorithm for simulating coupled Wright-Fisher diffusions}\label{sec:exact_alg_coupled}
%\section{Exact rejection algorithm for simulating coupled Wright-Fisher diffusions}\label{sec:exact_alg_coupled}
Let $X_t$ be the $n$-dimensional vector of allele frequencies satisfying (\ref{eq:coupled_WF_general}). 
Following the same scheme as Algorithm \ref{alg:Beskos} in Section \ref{sec:background}, 
Algorithm \ref{alg:CWFdiff} simulates exact skeletons of paths of 
coupled Wright-Fisher diffusions with $L$ loci and $d_i$ allele types each, $i\in\{1,\ldots,L\}$. 
Candidate processes in this case are $L$ independent $(d_i-1)$-dimensional neutral Wright-Fisher diffusions, 
each one unique weak solution of (\ref{eq:candidate_Mi}) and sampled following Algorithm 
\ref{alg:multidim_neutralWF}. 

The exact rejection algorithm proposed in this paper relies on the existence and characterization 
of the following acceptance-rejection probability, which is detailed in Theorem \ref{th:ESim_CWF} 
and whose proof is deferred to Section \ref{sec:proofs}.
\begin{thm}\label{th:ESim_CWF}
 Let $\mathbb{CWF}_{\alpha,G,x_0}$ be the law of $X$, the solution of
 (\ref{eq:coupled_WF_general}) and $\mathbb{WF}L_{\alpha,x_0}$ be the joint law of $L$ 
 independent $(d_i-1)$-dimensional neutral Wright-Fisher diffusions solution of (\ref{eq:candidate_Mi}), 
 $i\in\{1,\ldots, L\}$. 
 Then, the Radon-Nykod\'{y}m derivative of $\mathbb{CWF}_{\alpha,G,x_0}$ w.r.t. $\mathbb{WF}L_{\alpha,x_0}$ 
 is of the form
\begin{equation}\label{eq:RadonNykodym-CWF}
 \dfrac{d\mathbb{CWF}_{\alpha,G,x_0}}{d\mathbb{WF}L_{\alpha,x_0}}
 =\exp\left\{A(X_0,X_T)\right\}\exp\left\{-\int_0^T \phi(X_t) dt\right\}\\
\end{equation}
and there exist constants $A^-, A^+, C^-$ and $C^+$ such that $A(X_0,X_T)$ is 
bounded on $[0,1]^n\times[0,1]^n$ by $A^-\leq A(X_0,X_T)\leq A^+$ 
and $\phi(X_t)$ is bounded on $[0,1]^n$ by $C^-\leq\phi(X_t)\leq C^+$,
with
\begin{multline*}
 A(X_0,X_T):=\displaystyle\int_0^T V(X_t)\cdot dX_t= \sum_{i=1}^{L}\sum_{j=1}^{d_i-1} 
 \bigg(K^{ij}_s (X^{ij}_T-X^{ij}_0)
 +\sum_{l=1}^{L}K^{ij}_l (X^{ij}_T-X^{ij}_0)\\
 +\displaystyle\sum_{l=i+1}^{L}\sum_{k=1}^{d_l-1}  
 K^{ij}_{lk}\ (X^{ij}_TX^{lk}_T-X^{ij}_0X^{lk}_0)\bigg)
\end{multline*}
and
$$
\phi(X_t):=\dfrac{1}{2}\left[ (V(X_t))^T D(X_t) V(X_t)
      +2 (V(X_t))^T \alpha(X_t)\right].
$$
\end{thm}
Using Algorithm \ref{alg:multidim_neutralWF} in Step 3 (or in case $d_i=2$ for some $i$, 
the corresponding modification for one-dimensional diffusions), Algorithm \ref{alg:CWFdiff} 
returns an exactly simulated skeleton of the solution of (\ref{eq:coupled_WF_general}).

\begin{algorithm}
\caption{Exact rejection algorithm for simulating skeletons of the paths $(X_t)_{t\in[0,T]}$ 
of a diffusion process with law 
$\mathbb{CWF}_{\alpha, G, x_0}$}\label{alg:CWFdiff}
\algsetup{linenodelimiter=}
\begin{algorithmic}[1] % enter the algorithmic environment
    \STATE Simulate $\mathbf{\Phi}$, a Poisson process on $[0,T]\times[0,C^+-C^-]$
    \STATE Simulate $U\sim \mbox{Uniform}(0,1)$
    \STATE Given $\mathbf{\Phi}=\{(t_j,\psi_j): j=1,\ldots,J\}$, simulate $X\sim \mathbb{WF}L_{\alpha,x_0}$ 
    at times $\{t_1,\ldots,t_J, T\}$.
    \IF{$\phi(X_{t_j})-C^-\leq \psi_j,\ \forall j$ and $U\leq \exp\{A(X_0, X_T)-A^+\}$}
        \RETURN $\{(t_j,X_{t_j}),\ \forall j\}\cup \{(T,X_T)\}$
   \ELSE \STATE Go back to Step 1. 
    \ENDIF
\end{algorithmic}
\end{algorithm}
The algorithm's computational complexity can also be established, and is made precise in Proposition \ref{prop:complexity}, whose proof can be found in Section \ref{sec:proofs}. 
\begin{prop}\label{prop:complexity}
 Let $L$ be the number of loci, $M(t)$ denote the total number of coefficients 
 that must be computed in the implementation of Algorithm \ref{alg:deathprocess}, 
 where $t\in(0,T)$ is the time distance between two sampled skeleton points, and 
 let $N(T)$ denote the number of Poisson points required until the first skeleton in 
 Algorithm \ref{alg:CWFdiff} is accepted. Then, $\text{E}[LM(t)]<\infty$ 
 and $\text{E}[N(T)]<\infty$, and more specifically, there exists $\kappa>0$ such that 
 $$
 \text{E}[LM(t)]=o(t^{-(1+\kappa)}) \text{ as } t\to 0, \text{ and } 
 \text{E}[N(T)]\leq T(C^+-C^-)e^{T(C^+-C^-)+A^+-A^-}.
 $$
\end{prop}
In summary, the complexity of Algorithm \ref{alg:CWFdiff} increases either as $t\to 0$, when the average number of coefficients to be computed in Algorithm \ref{alg:deathprocess} increases as $1/t$, or with increasing $T$, when the average number of Poisson points needed until acceptance increases exponentially. 

The latter is easily solvable, simply by considering shorter intervals $[t_{k-1}, t_k]$ such that 
$$\bigcup_{k=1}^K[t_{k-1}, t_k]=[0,T], \text{ with } t_0=0, t_K=T, \text{ and }t_{k-1}<t_{k}, \ \forall k\in\{1, \ldots, K\},$$
and then concatenating the accepted skeletons in each of them. To solve the problem when $t\to 0$, we followed the recommendation in \cite{Jenkins2017}, and whenever $t<0.05$, resort to the approximation in (\ref{eq:asymp_aprox}). 

While asymptotically, the algorithm's growth rate does not depend on the number of loci,
%the number of loci $L$ do not play a relevant role in the algorithm's complexity, 
it is worth mentioning that with increasing $L$ the acceptance probability decreases, as there are a larger number of skeletons that need to be accepted simultaneously, which naturally affects the algorithm's feasibility. This will be clearer in the next section, where simulation results for examples with $L=2$ and $L=4$ are provided. 
As expected, the acceptance probability also decreases whenever the target diffusion differs more from the 
neutral Wright-Fisher candidate. This is exemplified in the next section, where 
results are shown for two coupled Wright-Fisher models with the same number of loci and mutation parameters, 
but different selective pairwise interactions.

\section{Numerical experiments}\label{sec:num_exp}
% %\section{Numerical experiments}\label{sec:num_exp}
%\input{./Sections/NumericalExps.tex}
In the following, several implementations of Algorithm \ref{alg:CWFdiff} are shown along with their simulation results. 
 The examples below 
represent plausible network structures present in biological applications, and whose interaction parameters are of interest. Examples of these include frequency-dependent selection in population dynamics with applications to vaccine's interventions \cite{Corander2017}, genome-wide discovery of interdependent loci affecting
antibiotic resistance \cite{Schubert2019} or analysis of sequence data \cite{Gao2018}. 

%\subsection{Two loci with two allele types}
Consider the case of two loci with two allele types each, i.e., $L=2$ and $d_1=d_2=2$. 
A particular example with one type allele interaction between loci with no within-locus selection reduces
(\ref{eq:coupled_WF_general}) to
%\begin{align}\label{eq:twodim}
\begin{equation}\label{eq:twodim}
  \begin{cases}
  dX_t^{11}=  \alpha^{11}(X_t^{11}) dt+X_t^{11}(1-X_t^{11})hX_t^{21}  dt+\sqrt{X_t^{11}(1-X_t^{11})}dB_t^1\\
  dX_t^{21}= \alpha^{21}(X_t^{21}) dt+ X_t^{21}(1-X_t^{21})hX_t^{11}  dt+\sqrt{X_t^{21}(1-X_t^{21})}dB_t^2
  \end{cases}
  \end{equation}
%\end{align}
where the $\alpha^{i1}(\cdot)$ are as in (\ref{eq:drift_coupled_WF_general}), $B=(B_t^1,B_t^2)$ is a vector of independent Brownian motions, $H^{12}=H^{21}=\begin{pmatrix} h&0 \\ 0&0\end{pmatrix}$, and $H$ is $0$ everywhere else. 

In this case,
  $A(X_0,X_T):=h(X_T^{11}X_T^{21}-X_0^{11}X_0^{21})$ and
  \begin{equation*}
   \begin{multlined}[\textwidth]
\phi(X_t)=\frac{1}{2}\left(h^2 \left[ (X_t^{21})^2 X_t^{11}(1-X_t^{11}) + (X_t^{11})^2 X_t^{21}(1-X_t^{21})\right]\right.{}\\
     \left.{}+ 2h \left[X_t^{21}\alpha^{11}(X_t^{11}) + X_t^{11}\alpha^{21}(X_t^{21})\right]\right),  
  \end{multlined}
  \end{equation*}
and the bound constants were set to 
$A^+=h(1-x_0^{11}x_0^{21})$, $C^-=-\frac{h}{2}(|\theta^{1}|+|\theta^{2}|)$
and $C^+=\frac{h}{2}(\frac{h}{2}+\theta^{1}_{1}+\theta^{2}_{1})$, where $|\theta^i|=\theta_{1}^{i}+\theta_{2}^{i}$. 

Results of several simulation scenarios for sampled skeletons of paths solution of (\ref{eq:twodim}) are shown in Table 1 and Table 2, where the total length $T$ of the considered interval $[0,T]$, the initialization of the path $(x_0^{11}, x_0^{21})$, the average number of attempts (drawn skeletons) until acceptance, average number of Poisson points needed until acceptance, average number of coefficients computed in Algorithm \ref{alg:deathprocess}, the acceptance probability, and the number of approximations needed due to small $t$s in between drawn points of the skeleton and the average time in seconds per accepted path are reported.

As shown in Table \ref{tab:h01} and Table \ref{tab:h05}, the average number of coefficients needed in shorter intervals where the sampled Poisson points are more likely to be close to each other ($t\to 0$) are larger than for longer intervals, as was expected from the results presented in Proposition \ref{prop:complexity}. Also, the acceptance probabilities when $h=1$ drop to around half compared with the simulations when $h=0.1$. This is also expected, as the model with larger pairwise interaction parameter differs more from the candidate paths, so acceptance of the candidate becomes harder. This is also reflected in the average number of attempts, needed Poisson points and coefficients, which increase consistently in the case $h=1$. Running time also increases with increasing $T$. In the case for $h=1$ and $T=5$ the total running times became prohibitive.

\begin{table}
\centering
\caption{Table for $10,000$ sampled paths satisfying (\ref{eq:twodim}) with $h=0.1$ and $\theta^{1}_{1}=\theta^{1}_{2}=\theta^{2}_{1}=\theta^{2}_{2}=0.01$}
\label{tab:h01}
\begin{tabular}{cccccccc} 
\multicolumn{8}{c}{$h=0.1$}\\
\hline
$T$ & $(x_0^{11}, x_0^{21})$ & Att. & Poisson points &Coeffs& Approx&Acc prob&Time (s)\\
\hline
0.1 & (0.5, 0.5) & 1.07 &0.001 & 300.68 &0.001 &0.93 &0.146\\ 
0.1& (0.02, 0.8) & 1.11 & 0.001 &311.04 &0.001 & 0.90&0.149\\
1.0 & (0.5, 0.5) & 1.08 &0.004 &7.79 & 0.001&0.92& 0.001\\ 
1.0& (0.02, 0.8) & 1.08 & 0.004 &7.78 &0.001 &0.92& 0.001\\
5.0 & (0.5, 0.5) & 1.09 &0.029 &3.78 &0.002 &0.92& 0.001\\ 
5.0& (0.02, 0.8) & 1.12 & 0.031 &4.23 &0.004 &0.89& 0.001\\
\hline
\end{tabular}
\end{table}
\begin{table}
\centering
\caption{Table for $10,000$ sampled paths satisfying (\ref{eq:twodim}) with $h=1$ and $\theta^{1}_{1}=\theta^{1}_{2}=\theta^{2}_{1}=\theta^{2}_{2}=0.01$}
\label{tab:h05}
\begin{tabular}{cccccccc}  
\multicolumn{8}{c}{$h=1$}\\
\hline
$T$ & $(x_0^{11}, x_0^{21})$ & Att. & Poisson points &Coeffs& Approx&Acc prob &Time (s)\\
\hline
0.1 & (0.5, 0.5) & 2.11 & 0.059 &610.96 & 0.005&0.47 & 0.318\\
0.1 & (0.02, 0.8) &  2.69 &0.076 &778.90 & 0.005&0.37 & 0.393\\
1.0 & (0.5, 0.5) &  2.23 & 0.615 &90.19 & 0.005&0.45 & 0.058\\
1.0 & (0.02, 0.8) &  2.77 & 0.774 &111.10 & 0.001&0.36 & 0.063\\
\hline
\end{tabular}
\end{table}

In order to establish the correctness of Algorithm \ref{alg:CWFdiff} and with the aim of providing 
a qualitative comparison, samples from the paths of (\ref{eq:twodim}) at a large $T$ were compared with 
the corresponding stationary density, and this was done for different mutation parameters (see Figure 1 and 
Figure 2) showing satisfactory results. Note that the stationary density for (\ref{eq:twodim}) 
can be explicitly computed and its normalizing constant reads, see \cite{Aurell2019a}, %\cite{Aurell2019},
\begin{equation}\label{eq:norm_constant}
Z=\frac{\Gamma(\theta^2_1)\Gamma(\theta^2_2)}{\Gamma(\theta^2_1+ \theta^2_2)}\sum_{n=0}^{\infty}\frac{(\theta^2_1)^{(n)}(2h)^n}{|\theta^2|^{(n)}n!} \frac{\Gamma(\theta^1_2)}{\Gamma(|\theta^1|+n)}\Gamma(\theta^1_1+n), 
\end{equation}
where $a^{(n)}=a(a+1)\ldots(a+n-1)$. For the qualitative comparisons in Figure 1 and Figure 2, the infinite sum is truncated at a sufficiently high $n$ that the remainder is negligible. For the parameters in Table \ref{tab:h01} and \ref{tab:h05} the truncation level is set to $n=70$. 

\begin{figure}\label{fig:histograms1}
 \centering
    \begin{tikzpicture}
        \begin{axis}[
            width=.40\textwidth,
            height=.40\textwidth,
            view={110}{30},
            colormap/blackwhite,
            axis x line = bottom,
            axis y line = left,
            axis z line = left,
            xtick = {0,0.5,1},
            ytick = {0,0.5,1},
            ztick = {0,0.05,0.1,0.15},
            zticklabels={},
            xlabel = {$x_1$},
            ylabel = {$x_2$},
            xmin=-0.01, xmax=1.01,
            ymin=-0.01, ymax=1.01,
        ]
            \addplot3[surf,mesh/cols=50,mesh/ordering=x varies,z buffer=sort,shader=interp,
            %restrict z to domain=0:0.4
            ] table {fig/h01_theta001_stationary_density.txt};
        \end{axis}
    \end{tikzpicture}
\input{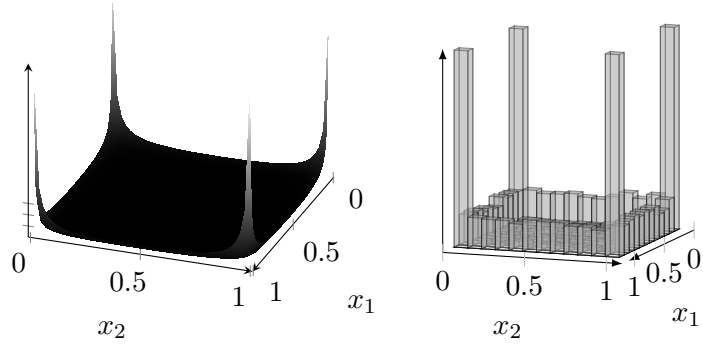}
\caption{Stationary distribution (left) and histogram of 10,000 samples (right) from $(X_t^{11},X_t^{21})$ satisfying (\ref{eq:twodim}) at $T=5$ with $(x_0^{11},x_0^{21})=(0.5, 0.5)$, $h=0.1$ and 
 $\theta^{i}_j=0.01$}
\end{figure}
\begin{figure}\label{fig:histograms2}
 \centering
    \begin{tikzpicture}
        \begin{axis}[
            width=.40\textwidth,
            height=.40\textwidth,
            view={110}{30},
            colormap/blackwhite,
            axis x line = bottom,
            axis y line = left,
            axis z line = left,
            xtick = {0,0.5,1},
            ytick = {0,0.5,1},
            ztick = {0,0.05,0.1,0.15},
            zticklabels={},
            xlabel = {$x_1$},
            ylabel = {$x_2$},
            xmin=-0.01, xmax=1.01,
            ymin=-0.01, ymax=1.01,
        ]
           \addplot3[surf,mesh/cols=50,mesh/ordering=x varies,z buffer=sort,shader=interp,
            ] table {fig/h01_theta1208_stationary_density.txt};
        \end{axis}
    \end{tikzpicture}
   \input{fig/Simulation_CWF_2x2_T5_x05_theta1208_h01.RData.txt.tex}
 \caption{Stationary distribution (left) and histogram of 10,000 samples (right) from $(X_t^{11},X_t^{21})$ satisfying (\ref{eq:twodim}) at $T=5$ with $(x_0^{11},x_0^{21})=(0.5, 0.5)$, $h=0.1$ and 
 $\theta^{1}_j=1.2,\theta^{2}_j=0.8$}
\end{figure}

Consider now the case of four loci with two allele types each, i.e., $L=4$ and $d_i=2$ for $i\in\{1,2,3,4\}$. A particular example with one type allele 
interaction between loci with no within-locus selection reduces
(\ref{eq:coupled_WF_general}) to

 \begin{equation}\label{eq:fourdim}
 \begin{dcases}
 \begin{multlined}[0.8025\textwidth]\textstyle
dX_t^{11}=\alpha^{11}(X_t^{11}) dt+X_t^{11}(1-X_t^{11})(h_1X_t^{21}+h_2X_t^{31}h_3X_t^{41}) dt+{}\\\textstyle
 {}+\sqrt{X_t^{11}(1-X_t^{11})}dB_t^1  
 \end{multlined}\\
  \textstyle dX_t^{21}=\alpha^{21}(X_t^{21}) dt+ X_t^{21}(1-X_t^{21})h_1X_t^{11}  dt+\sqrt{X_t^{21}(1-X_t^{21})}dB_t^2\\
  \textstyle dX_t^{31}=\alpha^{31}(X_t^{31}) dt+ X_t^{31}(1-X_t^{31})h_2X_t^{11}  dt+\sqrt{X_t^{31}(1-X_t^{31})}dB_t^3\\
  \textstyle dX_t^{41}= \alpha^{41}(X_t^{41}) dt+ X_t^{41}(1-X_t^{41})h_3X_t^{11}  dt+\sqrt{X_t^{41}(1-X_t^{41})}dB_t^4\\
\end{dcases}
\end{equation}
where the $\alpha^{i1}(\cdot)$ are as in (\ref{eq:drift_coupled_WF_general}), $B=(B_t^1,B_t^2,B_t^3,B_t^4)$ is a vector of independent Brownian motions, 
$$H^{12}=H^{21}=\begin{pmatrix} h_1&0 \\ 0&0\end{pmatrix},\, H^{13}=H^{31}=\begin{pmatrix} h_2&0 \\ 0&0\end{pmatrix} \text{ and }
H^{14}=H^{41}=\begin{pmatrix} h_3&0 \\ 0&0\end{pmatrix},$$ and $H$ is $0$ everywhere else. In this case,
% \vspace{-4pt}
\begin{equation*}
   \begin{multlined}[\textwidth]
 A(X_0,X_T)=h_1(X_T^{11}X_T^{21}-X_0^{11}X_0^{21}]+h_2[X_T^{11}X_T^{31}-X_0^{11}X_0^{31}]{}\\{}
 +h_3[X_T^{11}X_T^{41}-X_0^{11}X_0^{41})
\end{multlined}
\end{equation*}
 and
\begin{equation*}
   \begin{multlined}[\textwidth]
 \phi(X_t)=\frac{1}{2}([(h_1)^2X_t^{21}(1-X_t^{21})
 +(h_2)^2X_t^{31}(1-X_t^{31})+(h_3)^2X_t^{41}(1-X_t^{41})](X_t^{11})^2\\
 +(h_1X_t^{21}+h_2X_t^{31}+h_3X_t^{41})^2X_t^{11}(1-X_t^{11})+ 2[(h_1X_t^{21}+h_2X_t^{31}+h_3X_t^{41})\alpha^{11}(X_t^{11})\\
 +(h_1\alpha^{21}(X_t^{21})+h_2\alpha^{31}(X_t^{31})+h_3\alpha^{41}(X_t^{41}))X_t^{11}]),
   \end{multlined}
\end{equation*}
and the bound constants were set to %$A^+=h_1(1-x_0^{11}x_0^{21})+h_2(1-x_0^{11}x_0^{31})+h_3(1-x_0^{11}x_0^{41}),$
% $C^-=-\frac{1}{2}(|\theta^1|(h_1+h_2+h_3)+|\theta^{2}|h_1+|\theta^{3}|h_2+|\theta^{4}|h_3)$ and \\
% $C^+=\frac{1}{2}(\theta^{1}_{1}(h_1+h_2+h_3)
% +\frac{(h_1+h_2+h_3)^2+(h_1)^2+(h_2)^2+(h_3)^2}{4}+h_1\theta^{1}_{2}+h_2\theta^{1}_{3}+h_3\theta^{1}_{4})$.
% \begin{equation*}
% \begin{multlined}[\textwidth]
\begin{align*}
 &\textstyle
A^+=h_1(1-x_0^{11}x_0^{21})+h_2(1-x_0^{11}x_0^{31})+h_3(1-x_0^{11}x_0^{41}),\\
&\textstyle C^-=-\frac{1}{2}(|\theta^1|(h_1+h_2+h_3)+|\theta^{2}|h_1+|\theta^{3}|h_2+|\theta^{4}|h_3) \text{ and }\\
&\textstyle C^+=\frac{1}{2}(\theta^{1}_{1}(h_1+h_2+h_3)
+\frac{(h_1+h_2+h_3)^2+(h_1)^2+(h_2)^2+(h_3)^2}{4}+h_1\theta^{1}_{2}+h_2\theta^{1}_{3}+h_3\theta^{1}_{4}).
\end{align*}
% \end{multlined}
% \end{equation*}
%\\
Similarly to the previous example with $L=2$ and interaction parameter $h=1$, model (\ref{eq:fourdim}) 
differs more from the candidate process, as say, a model with only one pairwise interaction parameter 
(that is, a model with, for example, $h_2=h_3=0$). This is clearly reflected in the average low acceptance 
probabilities, or equivalently, in the average number of attempts or simulated Poisson points needed until acceptance, see Table \ref{tab:h05_3}. Moreover, simulating model (\ref{eq:fourdim}) implicitly requires to simultaneously accept four candidate paths, which makes acceptance more difficult. Nonetheless, it is still feasible to use Algorithm \ref{alg:CWFdiff} in these scenarios, as other approximate simulation strategies would be affected by similar problems. 
\begin{table}
\centering
\caption{Table for $10,000$ sampled paths with $h_1=0.1, h_2=0.15, h_3=0.2$ and $\theta^{i}_{j}=0.2$}
\label{tab:h05_3}
\begin{tabular}{cccccccc} 
\multicolumn{8}{c}{$h_1=0.1, h_2=0.15, h_3=0.2$}\\
\hline
$T$ & $x_0^{i1}$ & Att. & Poisson points & Coeffs&Approx&Acc prob &Time (s)\\
\hline
0.1 & 0.5 & 1.45 & 0.054 &412.65 & 0.001 &0.69 & 0.402\\
1.0 & 0.5 & 2.08 & 0.811 & 112.60&0.005 &0.48 & 0.129\\ 
5.0 & 0.5 &9.96 & 19.648 & 881.37&0.004 &0.10&1.007 \\ %3the last one
\hline
\end{tabular}
\end{table}

\section{Simulation of multidimensional neutral Wright-Fisher bridges}\label{sec:bridge}
%\section{Simulation of multidimensional neutral Wright-Fisher bridges}\label{sec:bridge}
To complete our simulation scheme, this section presents an exact simulation technique for sampling from neutral 
$(d-1)$-dimensional Wright-Fisher bridges. As mentioned before, once Algorithm \ref{alg:CWFdiff}  recovers a skeleton 
of the desired coupled Wright-Fisher diffusion, the remaining of the path can be filled by sampling from the 
corresponding neutral Wright-Fisher bridges, with no further reference to the target distribution needed, see, for example, \cite{Beskos2006}.

Consider a $(d-1)$-dimensional Wright-Fisher bridge, between $x$ at time $0$ 
and $z$ at time $t$. Its transition density is given by, see \cite{Fitzsimmons1993},
\begin{equation}\label{eq:trans_multibridge}
g_{z,t}(x,y;s)=\frac{g(x,y;s)g(y,z;t-s)}{g(x,z;t)}, \quad 0<s<t,
\end{equation}
where $g(\cdot,\cdot;\cdot)$ is as in (\ref{eq:mult_trans_den}). The precise 
eigenfunction expansion for $g_{z,t}(x,\cdot;s)$ is provided in the following proposition, 
whose prove can be found in Section \ref{sec:proofs}.

\begin{prop}\label{prop:trans_multibridge}
Let $g_{z,t}(x,\cdot;s)$  be the transition density function of a $(d-1)$-dimensional Wright-Fisher bridge. Then, 
its eigenfunction expansion reads
 $$
 g_{z,t}(x,y;s)=\sum_{m=0}^{\infty}\sum_{n=0}^{\infty}
 \sum_{\underset{|l|=m}{l}}\sum_{\underset{|r|=n}{r}}
 p_{m,n,l,r}^{(x,z,s,t,\theta)}\mathcal{D}_{\theta+l+r}(y),
 $$
 with
 $$
 p_{m,n,l,r}^{(x,z,s,t,\theta)}=\frac{q_m^{\theta}(s)q_n^{\theta}(t-s)}{g(x,z;t)}
 \mathcal{M}_{m,x}(l)\mathcal{D}_{\theta+r}(z)
 \mathcal{DM}_{\theta+l;n}(r)
 $$
 where $\mathcal{DM}_{\theta+l;n}(\cdot)$ denotes the PMF of a Dirichlet-Multinomial random variable, with
 $$
 \mathcal{DM}_{\theta+l;n}(r)=\frac{n!\Gamma(|\theta+l|)}
 {\Gamma(|\theta+l+r|)}\prod_{j=1}^{d}\frac{\Gamma(\theta_j+l_j+r_j)}
 {l_j!\Gamma(\theta_j+l_j)}.
 $$
 \end{prop}
Following the result in Proposition \ref{prop:trans_multibridge}, a sampling scheme for $g_{z,t}(x,y;s)$ is provided in Algorithm \ref{alg:multidim_WFbridge}.
\begin{algorithm}
\caption{Exact simulation of samples from $g_{z,t}(x,\cdot;s)$, transition density of the $(d-1)$-dimensional 
neutral Wright-Fisher bridge}\label{alg:multidim_WFbridge}
\algsetup{linenodelimiter=}
\begin{algorithmic}[1] % enter the algorithmic environment
    \STATE Simulate $(M,N,L,R)\sim \{p_{m,n,l,r}^{(x,z,s,t,\theta)}: (m,n,l,r) \in \mathbb{N}\times \mathbb{N}\times \mathbb{N}^d \times \mathbb{N}^d\}$.
    \STATE Given $(M,N,L,R)=(m,n,l,r)$, simulate $Y\sim \text{Dirichlet}(\theta+l+r)$.
    \RETURN $Y=(y_1,\ldots, y_{d-1})$.
  \end{algorithmic}
\end{algorithm}

Similarly to Step 1 in Algorithm \ref{alg:multidim_neutralWF}, sampling exactly from the 
discrete random variable with PMF $p_{m,n,l,r}^{(x,z,s,t,\theta)}$ is not straightforward. 
However, given the results obtained so far, it only remains to find how to evaluate $g(x,z;t)$ 
without approximation error, and the sampling strategy will be complete. As pointed out in 
\cite{Jenkins2017} for the one-dimensional case, note that evaluating $g(x,z;t)$ at $x$ and $z$ 
is a different problem than sampling from it. 

By (\ref{eq:mult_trans_den}) and (\ref{eq:qm}), one obtains 
\begin{equation}\label{eq:gxz}
 g(x,z;t)=
 \sum_{m=0}^{\infty} \sum^\infty_{i=0}(-1)^{i} c_{m+i,m}^{(x,z,t,\theta)}(m)=
 \sum_{m=0}^{\infty} \sum^\infty_{i=0}(-1)^{i}b_{m+i}^{(t,\theta)}(m)\text{E}[\mathcal{D}_{\theta+L_m}(z)],
 \end{equation}
where $L_m\sim\text{Multinomial}(m,x)$. 

As in Section \ref{sec:candidate}, the aim is to find monotonically converging bounds on 
$p_{m,n,l,r}^{(x,z,s,t,\theta)}$ so that the alternating series method can be applied.

Let 
$$d_{2m}=\sum_{i=0}^{m}c_{m+i,m-i}^{(x,z,t,\theta)}(m) \text{ and }
d_{2m+1}=\sum_{i=0}^{m}c_{m+1+i,m-i}^{(x,z,t,\theta)}(m), \text{ for } m=0,1,\ldots,
$$
which rearranging the terms in (\ref{eq:gxz}), gives the alternating series 
\begin{equation}\label{eq:d_seq}
g(x,z;t)=\sum_{m=0}^{\infty} (d_{2m} - d_{2m+1})= d_0-d_1+d_2-\ldots
\end{equation}
Indeed, it can be proved that the terms $(d_i)_{i\geq 0}$ are monotonically
decreasing from a certain threshold that is characterized in the following results, 
a required condition to apply the alternating series method.

First, note that the strategy presented in \cite{Jenkins2017} for the one-dimensional case applies here almost \emph{mutatis mutandis}, with the exception of the terms involving $\text{E}[\mathcal{D}_{\theta+L_m}(z)]$ 
which by an analogous (generalized) argument can be shown to decrease in $m$, 
as shown in the following lemma, proved in Section \ref{sec:proofs}.
\begin{lem}\label{lem:decreas_expec}
 Let $L_m\sim\text{Multinomial}(m,x)$. Then for all $m\in\mathbb{N}$
 $$
 \text{E}[\mathcal{D}_{\theta+L_{m+1}}(z)]\leq \tilde{K}^{(\theta,x,z)} \text{E}[\mathcal{D}_{\theta+L_m}(z)],
 $$
 where
\begin{multline}\label{eq:K_constant}
 \tilde{K}^{(\theta,x,z)}=
\left(\dfrac{|\theta|}{\theta_d}\left(1-\sum_{j=1}^{d-1}z_j\right)\vee \dfrac{2(1+|\theta|)}{1-\sum_{j=1}^{d-1}z_j}\right)\left(1-\sum_{j=1}^{d-1}x_j\right)\\
+\sum_{j=1}^{d-1}\left(\dfrac{|\theta|}{\theta_j}z_j\vee \dfrac{2(1+|\theta|)}{z_j}\right)x_j.
\end{multline}
\end{lem}

Now, similarly to (\ref{eq:C_constant}), the following bound is defined
\begin{equation}\label{eq:E_constant}
E^{(t,\theta)}=\inf\left\{m\geq 0 : 2j\geq C_{m-j}^{(t,\theta)} \text{ for all } j=0,\ldots, m\right\},
\end{equation}
and used in the next Proposition \ref{prop:decreasinm} that fully characterizes the bound on $m$. 
The proof is again deferred to Section \ref{sec:proofs}.
\begin{prop}\label{prop:decreasinm}
 Let the sequence $(d_i)_{i\geq 0}$ be as defined in (\ref{eq:d_seq}), and consider the bounds 
 $E^{(t,\theta)}$, %$\tilde{K}^{\theta}$ 
 $D_{\epsilon}^{(t,\theta)}$ and $\tilde{K}^{(\theta,x,z)}$ as in (\ref{eq:E_constant}), 
 %(\ref{lem:decreas_expec}) 
 (\ref{eq:D_constant}), and (\ref{eq:K_constant}) respectively. Then, for $\epsilon\in(0,1)$
 $$
 d_{2m+2}<d_{2m+1}<d_{2m}
 $$
 for all $m\geq E^{(t,\theta)}\vee D_{\epsilon}^{(t,\theta)}\vee 2\tilde{K}^{(\theta,x,z)}/\epsilon$.
\end{prop}

Once the bound in Proposition \ref{prop:decreasinm} is established, exact simulation of 
$(d-1)$-dimensional Wright-Fisher bridges ($d>2$) is possible by setting
$$
\tilde{F}_{m,n,l,r}^{(s,t,\theta)}:= C_{m}^{(s,\theta)} \vee C_{n}^{(t-s,\theta)} \vee
E^{(t,\theta)}\vee D_{\epsilon}^{(t,\theta)}\vee 2\tilde{K}^{(\theta,x,z)}/\epsilon,
$$
which for $2u\geq \tilde{F}_{m,n,l,r}^{(s,t,\theta)}$ provides the monotonically 
converging bounds
$$
e_{m,n,l,r}(2u+1)<e_{m,n,l,r}(2u+3)<p_{m,n,l,r}^{(x,z,s,t,\theta)}< e_{m,n,l,r}(2u+2)
<e_{m,n,l,r}(2u),
$$
where
$$
e_{m,n,l,r}(u)= \frac{\sum^{u}_{i=0}(-1)^{i}b_{m+i}^{(s,\theta)}(m)
\sum^{u}_{i=0}(-1)^{i}b_{n+i}^{(t-s,\theta)}}{\sum^{u+1}_{i=0}(-1)^{i} d_i} 
\mathcal{M}_{m,x}(l)\mathcal{D}_{\theta+r}(z)\mathcal{DM}_{\theta+l;n}(r),
$$
see Proposition 4 in \cite{Jenkins2017}. 

To recover exact samples from 
$p_{m,n,l,r}^{(x,z,s,t,\theta)}$, consider the pairing bijective function 
$\Sigma:\mathbb{N}\rightarrow \mathbb{N}\times \mathbb{N}\times \mathbb{N}^d \times \mathbb{N}^d$
such that $\Sigma(j)=(m,n,l,r)$. Now, for each $j$ there exist $v_j$, 
elements of $\tilde{v}\in\mathbf{R}^{J+1}$ 
(i.e., $\tilde{v}=\{v_j\}_{j=0}^J$) such that
$$
R^-_{\tilde{v}}(J):=
\sum_{j=0}^J  e_{\Sigma(j)}(2v_j+1)
\leq \sum_{j=0}^J p_{\Sigma(j)}^{(x,z,s,t,\theta)}
\leq \sum_{j=0}^J  e_{\Sigma(j)}(2v_j):=R^+_{\tilde{v}}(J),
$$
providing an analogous setting as the one presented in Section \ref{sec:candidate}. 
The proposed exact sampling scheme can be found in the following Algorithm \ref{alg:pmf_multibridges}, 
see Algorithm 5 in \cite{Jenkins2017},
that we reproduce here for completeness. 

\begin{algorithm}
\caption{Exact simulation of samples from the discrete random variable with PMF 
$\{p_{m,n,l,r}^{(x,z,s,t,\theta)}: (m,n,l,r) \in \mathbb{N}\times \mathbb{N}\times \mathbb{N}^d \times \mathbb{N}^d\}$}\label{alg:pmf_multibridges}
\algsetup{linenodelimiter=}
\begin{algorithmic}[1] % enter the algorithmic environment
    \STATE Set $j \gets 0, v_0\gets 0, \tilde{v}\gets (v_0)$
    \STATE Simulate $U\sim \mbox{Uniform}(0,1)$
    \REPEAT
    %\FORALL{$j\in\{0,\ldots, J\}$} 
      \STATE Set $v_j\gets \lceil \tilde{F}_{\Sigma(j)}^{(s,t,\theta)}/2 \rceil$ 
    %\ENDFOR
    \WHILE{$R^-_{\tilde{v}}(j)<U<R^+_{\tilde{v}}(j)$} 
      \STATE Set $\tilde{v}\gets \tilde{v}+(1,\ldots,1)$
    \ENDWHILE
    \IF{$R^-_{\tilde{v}}(j)>U$}
      \RETURN $\Sigma(j)$
    \ELSIF{$R^+_{\tilde{v}}(j)<U$} 
     \STATE Set $\tilde{v}\gets (v_0,\ldots, v_j,0)$ 
     \STATE Set $j\gets j+1$
    \ENDIF
    \UNTIL false
  \end{algorithmic}
\end{algorithm}

Other approaches to exact simulation of one-dimensional Wright-Fisher bridges include that 
recently proposed 
in \cite{Griffiths2018} that restricts to the case where either  
$\theta_1$ or $\theta_2$ are 0, and one or both of $x$ and $z$ are $0$, which is not applicable here.

\section{Proofs}\label{sec:proofs}
%\section{Proofs}\label{sec:proofs}
\begin{proof} [Proof of Proposition \ref{prop:gradient}]
Let $V:\mathbf{R}^n \to \mathbf{R}^{n}$ be such that 
$
V:=\nabla_{x}(\overline{V}\circ f)(x)$ where recall that
$$
f^{ik}(x)=\overline{x}^{ik}=
\bigg\{
\begin{array}{cl}
 x^{ij},& j,k=\{1,\ldots, d_i-1\},\\
 1-\sum_{j=1}^{d_i-1}x^{ij},& k=d_i.\\
\end{array}
$$
Then $\forall \ i\in\{1,\ldots, L\},j\in\{1,\ldots,d_i-1\}$,
 \begin{equation*}
\begin{array}{ll}
V^{ij}(x)&=\displaystyle\frac{\partial (\overline{V}\circ f)(x)}{\partial x^{ij}}=
\displaystyle\frac{\partial}{\partial \overline{x}^{ij}}\overline{V}(\overline{x})-
\frac{\partial}{\partial \overline{x}^{id_i}}\overline{V}(\overline{x})\\
&=\displaystyle s^{ij} + \sum_{l=1}^{L} \sum_{k=1}^{d_l} h^{il}_{jk} \overline{x}^{lk} 
- s^{id_i}-\sum_{l=1}^{L}\sum_{k=1}^{d_l} h^{il}_{d_ik}\overline{x}^{lk},\\
%&=\displaystyle s^{ij} + \sum_{l=1}^{L} \sum_{k=1}^{d_l-1} h^{il}_{jk} \overline{x}^{lk}+(1-\sum_{k=1}^{d_l-1}x^{lk})
%- s^{id_i}-\sum_{l=1}^{L}\sum_{k=1}^{d_l-1} h^{il}_{d_ik}\overline{x}^{lk}+(1-\sum_{k=1}^{d_l-1}x^{lk}),\\
&=\displaystyle s^{ij}- s^{id_i} + \sum_{l=1}^{L} \bigg(h^{il}_{jd_l}-h^{il}_{d_id_l} +
\sum_{k=1}^{d_l-1} (h^{il}_{jk}-h^{il}_{jd_l}-h^{il}_{d_ik}+h^{il}_{d_id_l}) x^{lk}\bigg)\\
&=\displaystyle K_s^{ij}+\sum_{l=1}^{L}\bigg(K_l^{ij} +\sum_{k=1}^{d_l-1} K_{lk}^{ij} x^{lk}\bigg)
=\displaystyle K_s^{ij}+\sum_{\underset{l\neq i}{l=1}}^{L}\bigg(K_l^{ij} +\sum_{k=1}^{d_l-1} K_{lk}^{ij} x^{lk}\bigg),
\end{array}
\end{equation*}
where %we have used $\overline{x}^{lk}=x^{lk}, \forall k\in\{1,\ldots, d_l-1\}$ 
%and $\overline{x}^{ld_l}=1-\sum_{k=1}^{d_l-1}x^{lk}$, and 
the last equality holds because
whenever $i=l$, all entries of the blocks $H^{ii}$ are $0$.
\end{proof}

\begin{proof}[Proof of Theorem \ref{th:ESim_CWF}]
By definition of $A(\cdot)$, $\phi(\cdot)$, and $X_t$, 
\begin{align}\label{eq:Proof_RadonNykodym-CWF}
 &\exp\left\{A(X_0,X_T)\right\}\exp\left\{-\int_0^T \phi(X_t) dt\right\}=
 \\
 &=\exp\left\{\int_0^T V(X_t)\cdot dX_t
 -\int_0^T\dfrac{1}{2}\left[ (V(X_t))^T D(X_t) V(X_t)
      +2 (V(X_t))^T \alpha(X_t)\right]dt\right\}\nonumber\\
   &=\exp\left\{\int_0^T  D^{\frac{1}{2}}(X_t)V(X_t) \cdot dB_t
 -\int_0^T\dfrac{1}{2}\left[ (V(X_t))^T D(X_t) V(X_t) \right]dt\right\}.\nonumber    
\end{align}
%Let $\Sigma:=D^{\frac{1}{2}}$. Then because $\Sigma(X_t)$, $G(X_t)=D(X_t)V(X_t)$ and $\alpha(X_t)$ are
% \begin{equation}\label{eg:novikov}
% \displaystyle\int_0^T \bigg(||\Sigma^{-1}(X_t)G(X_t)||^2+2 \Sigma^{-1}(X_t)G(X_t)\cdot \Sigma^{-1}(X_t)\alpha(X_t) \bigg) dt<\infty,
% \end{equation}
Because $V(\cdot)$ and $D(\cdot)$ are 
continuous on $[0,1]^n$, there exist constants $C^-$ and $C^+$ such that
$$C^-\leq\dfrac{1}{2} \displaystyle V(X_t)^T D(X_t)V(X_t) \leq C^+, \quad \text{a.s.}$$
Consequently, 
\begin{equation*}\label{eg:novikov}
\displaystyle\int_0^T\dfrac{1}{2} (V(X_t))^T D(X_t) V(X_t)
      dt<\infty, \quad \text{a.s.,}
\end{equation*}
so Novikov's condition is fulfilled and (\ref{eq:Proof_RadonNykodym-CWF}) can be identified as a Girsanov transformation \cite{Karatzas1998} with Girsanov kernel $(V(X_t))^T D^{\frac{1}{2}}(X_t)$. 

Let $\mathbb{Q}$ be the probability measure with
\begin{equation*}
    \frac{d\mathbb{Q}}{d\mathbb{WF}L_{\alpha,x_0}} =\exp\left\{A(X_0,X_T)\right\}\exp\left\{-\int_0^T \phi(X_t) dt\right\}.
\end{equation*}
It follows that the law of $X$ under $\mathbb{Q}$ coincides with $\mathbb{CWF}_{\alpha,G,x_0}$. Indeed, by Girsanov's theorem 
\begin{equation*}
    \tilde{B}_t = B_t - \int_0^t D^{\frac{1}{2}}(X_s)V(X_s) ds,
\end{equation*}
is a $\mathbb{Q}$-Brownian motion and 
\begin{equation*}
    dX_t = [\alpha(X_t) + D(X_t)V(X_t)]dt+  D^{\frac{1}{2}}(X_t)d\tilde{B}_t \\
    =[\alpha(X_t) + G(X_t)] dt+  D^{\frac{1}{2}}(X_t)d\tilde{B}_t.
\end{equation*}

Now by Proposition \ref{prop:gradient}, and for $K_s,K_l, K_{lk}\in\mathbf{R}^n$
\begin{align*}
 &\displaystyle\int_0^T V(X_t)\cdot dX_t
 =\displaystyle\int_0^T
 \displaystyle \bigg( K_s+\sum_{l=1}^{L}\bigg(K_l +\sum_{k=1}^{d_l-1} K_{lk}X_t^{lk}\bigg)\bigg) \cdot dX_t\\
 =&\displaystyle \int_0^T K_s\cdot dX_t
 +\int_0^T\sum_{l=1}^{L} K_l\cdot dX_t
 +\int_0^T\sum_{l=1}^{L}\sum_{k=1}^{d_l-1}  
 K_{lk} X_t^{lk} \cdot dX_t\\
 =&\displaystyle \sum_{i=1}^{L}\sum_{j=1}^{d_i-1}\Bigg(\int_0^T K^{ij}_s dX^{ij}_t
 +\int_0^T \sum_{l=1}^{L}K^{ij}_l dX^{ij}_t
 +\displaystyle\int_0^T\sum_{l=i+1}^{L}\sum_{k=1}^{d_l-1}  
 K^{ij}_{lk}\ d(X_t^{lk}X^{ij}_t)\Bigg)\\
 \end{align*}
where the farmost right term comes from pairing terms of the form 
$$K^{ij}_{lk}X_t^{lk}dX^{ij}_t+K^{lk}_{ij}X_t^{ij}dX^{lk}_t=K^{ij}_{lk} d(X_t^{lk}X^{ij}_t),$$ and recalling that $K^{ij}_{lk}=K^{lk}_{ij}$, and $i\neq l$ prevents squared terms. Hence, 
\begin{multline*}
 \displaystyle\int_0^T V(X_t)\cdot dX_t
 =\displaystyle \sum_{i=1}^{L}\sum_{j=1}^{d_i-1} \bigg(K^{ij}_s (X^{ij}_T-X^{ij}_0)
 +\sum_{l=1}^{L}K^{ij}_l (X^{ij}_T-X^{ij}_0)\\
 +\displaystyle\sum_{l=i+1}^{L}\sum_{k=1}^{d_l-1}  
 K^{ij}_{lk}\ (X^{lk}_TX^{ij}_T-X^{lk}_0X^{ij}_0)\bigg).\\
 \end{multline*}
The fact that
\begin{multline*}
A(X_0,X_T):=\sum_{i=1}^{L}\sum_{j=1}^{d_i-1} \bigg(K^{ij}_s (X^{ij}_T-X^{ij}_0)
 +\sum_{l=1}^{L}K^{ij}_l (X^{ij}_T-X^{ij}_0)\\
 +\displaystyle\sum_{l=i+1}^{L}\sum_{k=1}^{d_l-1}  
 K^{ij}_{lk}\ (X^{ij}_TX^{lk}_T-X^{ij}_0X^{lk}_0)\bigg)
 \end{multline*}
is bounded follows immediately, concluding the proof.
\end{proof}

\begin{proof} [Proof of Proposition \ref{prop:complexity}]
Let $M(t)$ be the total number of coefficients that must be computed in Algorithm \ref{alg:deathprocess}, with 
$t\in(0,T)$ the time distance between two sampled skeleton points. 
By Proposition 5 $(iv)$ in \cite{Jenkins2017}, there exists a $\kappa>0$ such that 
$\text{E}[M(t)]=o(t^{-(1+\kappa)})$ as $t\to 0$,
and further random coefficients needed in Algorithm \ref{alg:multidim_neutralWF} do not add to the algorithms' 
complexity. Similarly, 
although our rejection scheme uses Algorithm \ref{alg:deathprocess} $L$ times,
$$
\text{E}[LM(t)]=L\text{E}[M(t)]=o(t^{-(1+\kappa)}),
$$
so the algorithm's complexity is proportional to $L$, but its growth rate as $t\to 0$ remains the same as with 
$L=1$.

Let now
$\epsilon:=d\mathbb{CWF}_{\alpha,G,x_0}/d\mathbb{WF}L_{\alpha,x_0}$  be 
the acceptance-rejection probability in Algorithm \ref{alg:CWFdiff}. Because $A^-\leq A(X_T)\leq A^+$ and 
$C^-\leq\phi(X_t)\leq C^+$,
\begin{multline*}
\epsilon
\propto\exp\left\{A(X_T)-A^+\right\}\exp\left\{-\int_0^T (\phi(X_t)- C^-)dt\right\}\\
\geq \exp\{-T(C^ +-C^-) + A^--A^+\}.
\end{multline*}
Let $D$ refer to the number of Poisson points needed to decide upon acceptance or rejection of a proposed path. Then, following Proposition 3 in \cite{Beskos2006}
$$
\text{E}[N(T)]=\text{E}[D]/\epsilon=T(C^+-C^-)/\epsilon\leq T(C^+-C^-)e^{T(C^+-C^-)+A^+-A^-},
$$
where the first equality follows from considering the expectation of the sum of all drawn Poisson points $\sum_{i=1}^I D_i$ over $I$ iterations of the algorithm until the first path is accepted. Conditioning first on $I$ and applying the law of iterated expectations, and then applying the law of total expectation, concludes the proof.
\end{proof}

\begin{proof} [Proof of Proposition \ref{prop:trans_multibridge}]
Let $g_{z,t}(x,y;s)$ be the transition density function of a $(d-1)$-dimensional Wright-Fisher bridge, between $x$ at 
time $0$ and $z$ at time $t$. By (\ref{eq:trans_multibridge}) and (\ref{eq:mult_trans_den})
\begin{multline*}
g_{z,t}(x,y;s)=\frac{g(x,y;s)g(y,z;t-s)}{g(x,z;t)}\\
=\frac{1}{g(x,z;t)}\sum_{m=0}^{\infty} q_m^{\theta}(s) 
\sum_{\underset{|l|=m}{l}} \mathcal{M}_{m,x}(l)\mathcal{D}_{\theta+l}(y)\sum_{n=0}^{\infty} q_n^{\theta}(t-s) \sum_{\underset{|r|=n}{r}} \mathcal{M}_{n,y}(r)\mathcal{D}_{\theta+r}(z)\\
=\sum_{m=0}^{\infty}\sum_{n=0}^{\infty}
\sum_{\underset{|l|=m}{l}}\sum_{\underset{|r|=n}{r}}\frac{q_m^{\theta}(s)q_n^{\theta}(t-s)}{g(x,z;t)}
\mathcal{M}_{m,x}(l)\mathcal{D}_{\theta+r}(z)
\mathcal{M}_{n,y}(r)\mathcal{D}_{\theta+l}(y).
\end{multline*}
By definition
 \begin{multline*}
\mathcal{M}_{n,y}(r)\mathcal{D}_{\theta+l}(y)=
\frac{n!}{\prod_{j=1}^{d}l_j!}(1-\sum_{j=1}^{d-1}y_j)^{l_{d}}\prod_{j=1}^{d-1}y_j^{r_j}\\
\times\frac{\Gamma(|\theta+l|)}{\prod_{j=1}^{d}\Gamma(\theta_j+l_j)}
 (1-\sum_{j=1}^{d-1}y_j)^{\theta_{d}+l_{d}-1}\prod_{j=1}^{d-1}y_j^{\theta_j+l_j-1}.
  \end{multline*}
Multiplying and dividing by
$
\dfrac{\Gamma(|\theta+l+r|)}{\prod_{j=1}^d \Gamma(\theta_j+l_j+r_j)}
$
and rearranging, shows that
 \begin{multline*}
 \frac{n!\Gamma(|\theta+l|)}{\Gamma(|\theta+l+r|)}
\prod_{j=1}^{d}\frac{\Gamma(\theta_j+l_j+r_j)}{\Gamma(\theta_j+l_j)r_j!}
\dfrac{\Gamma(|\theta+l+r|)}{\prod_{j=1}^d \Gamma(\theta_j+l_j+r_j)}
(1-\sum_{j=1}^{d-1}y_j)^{\theta_{d}+l_{d}+r_d-1}\\
\times\prod_{j=1}^{d-1}y_j^{\theta_j+l_j+r_j-1}
=\mathcal{DM}_{\theta+l;n}(r)\mathcal{D}_{\theta+l+r}(y),
 \end{multline*}
 Now, identifying the coefficients of $p_{m,n,l,r}^{(x,z,s,t,\theta)}$ the proof is complete.
\end{proof}

\begin{proof} [Proof of Lemma \ref{lem:decreas_expec}] 
For the sake of simplicity, we will first consider the case $d=3$ and show that by analogous arguments 
the results can be extended to the general $(d-1)$-dimensional case. 
First, note the indexes of the sum on the right hand side of
\begin{equation*}
\text{E}[\mathcal{D}_{\theta+L_{m}}(z)]=
\sum_{\underset{|l|=m}{l}}\Pr(L_{m}=l)\mathcal{D}_{\theta+l}(z),
\end{equation*}
can be seen as placed on the $(d-1)$-face of the $d$-simplex $\mathcal{S}_d=\{l\in\mathbf{R}^d\mid l_{j}\geq 0, 
\sum_{j=1}^{d}l_j=m\}$. For example, if $d=3$, we only need to consider indexes $l_1$ and $l_2$ 
(see Figure \ref{fig:simplex3d}). 
 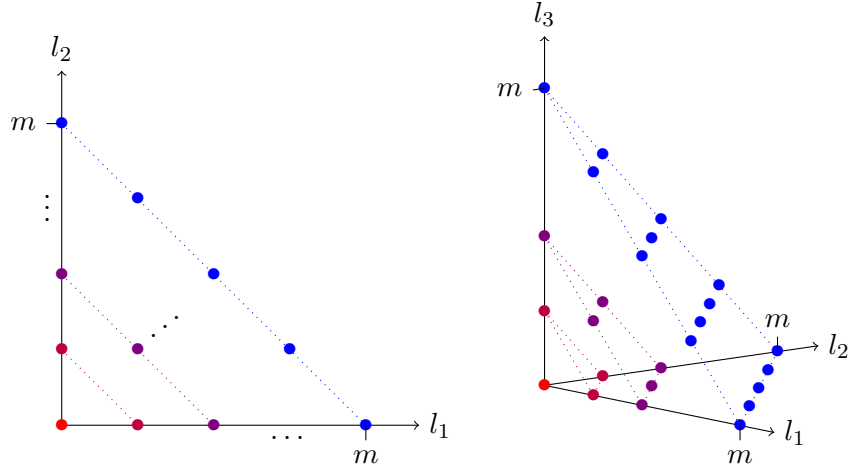
\begin{figure}
    % Set perspective for 3D rendering
    \def\maxm{2} \def\maxaxs{4.7}
    \def\eps{.2}
    \tdplotsetmaincoords{80}{50}
    \begin{center}
    \begin{tabular}{cc}
        \begin{tikzpicture}[]
            % Plot axis
            \draw [->] (0,0) -- (\maxaxs,0) node[anchor = west]{$l_1$};
            \draw [->] (0,0) -- (0,\maxaxs) node[anchor = south]{$l_2$};
            % For each total sum
            \foreach \m in {0,...,\maxm}{
                % Interpolate color (geeky)
                \pgfmathsetmacro{\c}{100/(\maxm+2)*\m}
                % Plot line l1+l2=m
                \draw[color=blue!\c!red,dotted] (0,\m) -- (\m,0);
                % Plot the integers using l2 = m - l1
                \foreach \lone in {0,...,\m}{
                    \node[color=blue!\c!red] at (\lone,\m-\lone) {$\bullet$};
                }
            }
            % Draw etc. dots
            \node[anchor=north] at (\maxm+1,0) {$\dots$};
            \node[anchor=south west] at (\maxm/2,\maxm/2) {
            $\iddots$};
            \node[anchor=east] at (0,\maxm+1) {$\vdots$};
            % Repeat for last value
            % Interpolate color (geeky)
            \pgfmathsetmacro{\m}{\maxm+2}
            \pgfmathsetmacro{\c}{100/(\maxm+2)*\m}
             % Plot line l1+l2=m
            \draw[color=blue!\c!red,dotted] (0,\m) -- (\m,0);
            % Plot the integers using l2 = m - l1
            \foreach \lone in {0,...,\m}{
                \node[color=blue!\c!red] at (\lone,\m-\lone) {$\bullet$};
            }
            % Show ticks indicating $m$ at the last value
            \draw (\maxm+2,0) -- (\maxm+2,-\eps) node[anchor = north]{$m$};
            \draw (0,\maxm+2) -- (-\eps,\maxm+2) node[anchor = east]{$m$};
        \end{tikzpicture}
        &
        \begin{tikzpicture}[tdplot_main_coords]
            % Plot axis
            \draw [->] (0,0,0) -- (\maxaxs,0,0) node[anchor = west]{$l_1$};
            \draw [->] (0,0,0) -- (0,\maxaxs,0) node[anchor = west]{$l_2$}; 
            \draw [->] (0,0,0) -- (0,0,\maxaxs) node[anchor = south]{$l_3$};
            % Get last value
            \pgfmathsetmacro{\lastm}{\maxm+2}
            % For each total sum
            \foreach \m in {0,...,\maxm,\lastm}{
                % Interpolate color (geeky)
                \pgfmathsetmacro{\c}{100/(\maxm+2)*\m}
                % Plot line l1+l2=m
                \draw[color=blue!\c!red,dotted] (0,\m,0) -- (\m,0,0);
                % Plot line l1+l3=m
                \draw[color=blue!\c!red,dotted] (0,0,\m) -- (\m,0,0);
                % Plot line l2+l3=m
                \draw[color=blue!\c!red,dotted] (0,0,\m) -- (0,\m,0);
                % Plot integers
                \foreach \l [evaluate=\l] in {0,...,\m}{
                    \pgfmathsetmacro{\lim}{\m-\l}
                    \foreach \p in {0,...,\lim}{
                        % Using l1 + l2 + l3 = m
                        \node[color=blue!\c!red] at (\l,\p,\m-\l-\p) {$\bullet$};
                    }
                }
            }
            % Show ticks indicating $m$ at the last value
            \draw (\maxm+2,0,0) -- ++ (0,0,-\eps) node[anchor = north]{$m$};
            \draw (0,\maxm+2,0) -- ++ (0,0,\eps) node[anchor = south]{$m$};
            \draw (0,0,\maxm+2) -- ++ (0,-\eps,0) node[anchor=east] {$m$};
        \end{tikzpicture}
        %\\
        %2D representing 3D & 3D
    \end{tabular}
\end{center}
    \caption{\label{fig:simplex3d} Graphical representation of the $2$-dimensional case ($d=3$). 
    Each point in the figure represents an element of the sum in (\ref{eq:multidim_sum}). For every and up to each $m$, 
    all points depicted in the plane $l_1-l_2$ (left) correspond to points depicted in the $m$th $(l_1-l_2-l_3)$-face (right), 
    which in turn represent each and every one of the summands in (\ref{eq:multidim_sum}).}
\end{figure}
Thus,
\begin{equation}\label{eq:multidim_sum}
\text{E}[\mathcal{D}_{\theta+L_{m}}(z)]=\sum_{l_1=0}^{m}\sum_{l_2=0}^{m-l_1}
\Pr(L_{m}=l)\mathcal{D}_{\theta+l}(z).
\end{equation}

Let us define the quantities
$$
%\mathcal{Q}_{m+1}:=\Pr(L_{m+1}=l)\mathcal{D}_{\theta+l}(z) \text{ and } 
\mathcal{Q}_{m}:=\Pr(L_{m}=l)\mathcal{D}_{\theta+l}(z) \text{ and } \mathcal{Q}_{m+1}^j:=\Pr(L_{m+1}=\tilde{l})\mathcal{D}_{\theta+\tilde{l}}(z),
$$
where $\tilde{l}\in\mathbf{R}^d$ is equal to $l$ except in its $j$th component where $\tilde{l}_j=l_j+1$. As a special case, when $j=d$, one can write $\tilde{l}_d=m-\sum_{j=1}^{d-1}l_j+1$.

Then, for $l_1\leq\lfloor mz_1 \rfloor, l_2\leq\lfloor mz_2 \rfloor$,
\begin{align}\label{eq:del1}
\mathcal{Q}_{m+1}^d &= 
\frac{m+1}{m+1-l_1-l_2}\frac{|\theta|+m}{\theta_3+m-l_1-l_2}(1-x_1-x_2)(1-z_1-z_2)\mathcal{Q}_{m}\\\nonumber
&\leq \frac{m+1}{1+m(1-z_1-z_2)}\frac{|\theta|+m}{\theta_3+m(1-z_1-z_2)}(1-x_1-x_2)(1-z_1-z_2)\mathcal{Q}_{m}\\\nonumber
&\leq \left(\dfrac{|\theta|}{\theta_3}\vee \frac{m^2+m(1+|\theta|)+|\theta|}{m^2(1-z_1-z_2)^2}\right)(1-x_1-x_2)(1-z_1-z_2)\mathcal{Q}_{m}\\
&\leq \left(\dfrac{|\theta|}{\theta_3}(1-z_1-z_2)\vee \dfrac{2(1+|\theta|)}{1-z_1-z_2}\right)(1-x_1-x_2)\mathcal{Q}_{m}\nonumber,
\end{align}
where from the second inequality, note that the function $g(m):=
\frac{m^2+m(1+|\theta|)+|\theta|}{m^2(1-z_1-z_2)^2}$ is decreasing in $m$, and thus for $m\geq 1$, it attains its maximum value 
at $g(1)$. Then, if $f(m):= \frac{m+1}{1+m(1-z_1-z_2)}\frac{|\theta|+m}{\theta_3+m(1-z_1-z_2)}$, one obtains $f(m)\leq f(0)\vee g(m)\leq f(0)\vee g(1)$ 
yielding the desired result.

Similarly, for $l_1\leq\lfloor mz_1 \rfloor, \lfloor mz_2 \rfloor\leq l_2$,
\begin{align}\label{eq:del2}
\mathcal{Q}_{m+1}^2 &= 
\frac{m+1}{l_2+1}\frac{|\theta|+m}{\theta_2+l_2}x_2z_2\mathcal{Q}_{m}
\leq \frac{m+1}{mz_2+1}\frac{|\theta|+m}{\theta_2+mz_2}x_2z_2\mathcal{Q}_{m}\\\nonumber&
\leq \left(\dfrac{|\theta|}{\theta_2}\vee\frac{m^2+m(1+|\theta|)+|\theta|}{m^2z_2^2}\right)x_2z_2\mathcal{Q}_{m}
\leq \left(\dfrac{|\theta|}{\theta_2}z_2\vee \dfrac{2(1+|\theta|)}{z_2}\right)x_2\mathcal{Q}_{m},
\end{align}
and for $\lfloor mz_1 \rfloor\leq l_1$ and all $l_2$,
\begin{align}\label{eq:del3}
\mathcal{Q}_{m+1}^1 &= 
\frac{m+1}{l_1+1}\frac{|\theta|+m}{\theta_1+l_1}x_1z_1\mathcal{Q}_{m}
\leq \frac{m+1}{mz_1+1}\frac{|\theta|+m}{\theta_1+mz_1}x_1z_1\mathcal{Q}_{m}\\\nonumber&
\leq \left(\dfrac{|\theta|}{\theta_1}\vee\frac{m^2+m(1+|\theta|)+|\theta|}{m^2z_1^2}\right)x_1z_1\mathcal{Q}_{m}
\leq \left(\dfrac{|\theta|}{\theta_1}z_1\vee \dfrac{2(1+|\theta|)}{z_1}\right)x_1\mathcal{Q}_{m}.
\end{align}

Combining the inequalities in (\ref{eq:del1}), (\ref{eq:del2}) and (\ref{eq:del3}), 
\begin{align*}
&\text{E}[\mathcal{D}_{\theta+L_{m+1}}(z)]=\sum_{l_1=0}^{\lfloor mz_1 \rfloor}\sum_{l_2=0}^{\lfloor mz_2 \rfloor}
\mathcal{Q}_{m+1}^d + \sum_{l_1=0}^{\lfloor mz_1 \rfloor}\sum_{l_2=\lfloor mz_2\rfloor+1}^{m+1-l_1} \mathcal{Q}_{m+1}^2+ \sum_{l_1=\lfloor mz_1 \rfloor+1}^{m+1}\!\!\sum_{l_2=0}^{m+1-l_1} \mathcal{Q}_{m+1}^1\\
&\leq \left(\dfrac{|\theta|}{\theta_3}(1-z_1-z_2)\vee \dfrac{2(1+|\theta|)}{1-z_1-z_2}\right)(1-x_1-x_2)\sum_{l_1=0}^{\lfloor mz_1 \rfloor}\sum_{l_2=0}^{\lfloor mz_2 \rfloor} \mathcal{Q}_{m}\\
&+\left(\dfrac{|\theta|}{\theta_2}z_2\vee \dfrac{2(1+|\theta|)}{z_2}\right)x_2\sum_{l_1=0}^{\lfloor mz_1 \rfloor}\sum_{l_2=\lfloor mz_2 \rfloor}^{m-l_1} \mathcal{Q}_{m}\\
&+\left(\dfrac{|\theta|}{\theta_1}z_1\vee \dfrac{2(1+|\theta|)}{z_1}\right)x_1\sum_{l_1=\lfloor mz_1 \rfloor}^{m}\sum_{l_2=0}^{m-l_1} \mathcal{Q}_{m}\\
&\begin{multlined}[0.955\textwidth]\textstyle
 \leq\left[\left(\dfrac{|\theta|}{\theta_3}(1-z_1-z_2)\vee \dfrac{2(1+|\theta|)}{1-z_1-z_2}\right)(1-x_1-x_2)\right.\\
\textstyle{}+\left.\displaystyle\sum_{j=1}^2\left(\dfrac{|\theta|}{\theta_j}z_j\vee \dfrac{2(1+|\theta|)}{z_j}\right)x_j\right]
\text{E}[\mathcal{D}_{\theta+L_{m}}(z)],
\end{multlined}%\sum_{l_1=0}^{m}\sum_{l_2=0}^{m-l_1}\mathcal{Q}_{m},
\end{align*}
where in the first inequality the terms of the sums starting in $l_j=\lfloor mz_j \rfloor+1$ are shifted by one index, 
and the last inequality holds after taking common factors and noting that the terms for 
$l_j=\lfloor mz_j \rfloor$ are bounded by both 
$$\left(\dfrac{|\theta|}{\theta_3}(1-z_1-z_2)\vee \dfrac{2(1+|\theta|)}{1-z_1-z_2}\right)(1-x_1-x_2)
\text{ and }\sum_{j=1}^2\left(\dfrac{|\theta|}{\theta_j}z_j\vee \dfrac{2(1+|\theta|)}{z_j}\right)x_j.$$

The proof for the general $(d-1)$-dimensional case follows analogously. Consider
\begin{equation*}
\text{E}[\mathcal{D}_{\theta+L_{m}}(z)]=\sum_{l_1=0}^{m}\sum_{l_2=0}^{m-l_1}\ldots
\sum_{l_{d-1}=0}^{m-|l|_{d-2}}\Pr(L_{m}=l)\mathcal{D}_{\theta+l}(z),
\end{equation*}
where $|l|_{d-2}=\sum_{j=1}^{d-2} l_j$.

Following the strategy used above, the sums can be partitioned in terms such that either 
\begin{itemize}
 \item [a)] $l_j\leq \lfloor mz_j \rfloor, \forall \ j$,
 \item [b)] $\lfloor mz_j \rfloor\leq l_j, j\neq 1$ and $l_i\leq \lfloor mz_i \rfloor, \forall i\neq j$ or
 \item [c)] $\lfloor mz_1 \rfloor\leq l_1$ and $l_i, \forall i\neq 1$ free.
\end{itemize}
Note that this partition covers all (non exclusive) combinations and includes all the elements of the sum. Now, comparing 
$\text{E}[\mathcal{D}_{\theta+L_{m+1}}(z)]$ with $\text{E}[\mathcal{D}_{\theta+L_{m}}(z)]$, 
the bounding constants for each case are
$$
\text{a)} \left(\dfrac{|\theta|}{\theta_d}\left(1-\sum_{j=1}^{d-1}z_j\right)\vee \dfrac{2(1+|\theta|)}{1-\sum_{j=1}^{d-1}z_j}\right)\left(1-\sum_{j=1}^{d-1}x_j\right),
\text{b)} \left(\dfrac{|\theta|}{\theta_j}z_j\vee \dfrac{2(1+|\theta|)}{z_j}\right)x_j,$$
and
$$
\text{c)}\left(\dfrac{|\theta|}{\theta_1}z_1\vee \dfrac{2(1+|\theta|)}{z_1}\right)x_1,
$$
yielding
$$
\text{E}[\mathcal{D}_{\theta+L_{m+1}}(z)]\leq \tilde{K}^{(\theta,x,z)} \text{E}[\mathcal{D}_{\theta+L_{m}}(z)],
$$
with
\begin{multline}
 \tilde{K}^{(\theta,x,z)}=
\left(\dfrac{|\theta|}{\theta_d}\left(1-\sum_{j=1}^{d-1}z_j\right)\vee \dfrac{2(1+|\theta|)}{1-\sum_{j=1}^{d-1}z_j}\right)\left(1-\sum_{j=1}^{d-1}x_j\right)\\
+\sum_{j=1}^{d-1}\left(\dfrac{|\theta|}{\theta_j}z_j\vee \dfrac{2(1+|\theta|)}{z_j}\right)x_j.
\end{multline}
\end{proof}
\begin{proof} [Proof of Proposition \ref{prop:decreasinm}]
The proof follows from that of Proposition 3 in \cite{Jenkins2017}, which is reproduced here for completeness. 

The inequality $d_{2m+1}<d_{2m}$ follows because if $m\geq E^{(t,\theta)}$ then $2j\geq C_{m-j}^{(t,\theta)}$ for all $j=0,\ldots, m$, which, by Proposition \ref{prop:bound_m}, implies %$b^{(t,\theta)}_{m-j+2j+1}(m-j)<b^{(t,\theta)}_{m-j+2j}(m-j))$
$b^{(t,\theta)}_{m+j+1}(m-j)<b^{(t,\theta)}_{m+j}(m-j))$.
Multiplying by $\text{E}[\mathcal{D}_{\theta+L_{m-j}}(z)]$ and then summing over $j=0,\ldots,m$ gives
\begin{equation}\label{eq:sum_c_seqs}
d_{2m+1}=\sum_{j=0}^m c_{m+j+1,m-j}^{(x,z,t,\theta)}<\sum_{j=0}^m c_{m+j,m-j}^{(x,z,t,\theta)}=d_{2m}.
\end{equation}

Proving $d_{2m+2}<d_{2m+1}$ requires some extra steps. First, note that
$$
d_{2m+2}=d_{2(m+1)}=\sum_{r=0}^{m+1} c_{m+1+r,m+1-r}^{(x,z,t,\theta)}=\sum_{j=-1}^m c_{m+2+j,m-j}^{(x,z,t,\theta)},
$$
where $j=r-1$. Noting now that $2j+1>2j\geq C_{m-j}^{(t,\theta)} \text{ for all } j=0,\ldots, m$, which implies $b^{(t,\theta)}_{m+j+2}(m-j)<b^{(t,\theta)}_{m+j+1}(m-j)$, and using the same argument as in (\ref{eq:sum_c_seqs}) yields 
$$
\sum_{j=1}^m c_{m+j+2,m-j}^{(x,z,t,\theta)}<\sum_{j=1}^m c_{m+j+1,m-j}^{(x,z,t,\theta)},
$$
where the sum is taken only over $j=1,\ldots, m$ so that the remaining terms in $d_{2m+2}$ and $d_{2m+1}$ can be compared. Indeed, it only remains to prove that
$$
c_{m+1,m+1}^{(x,z,t,\theta)}+c_{m+2,m}^{(x,z,t,\theta)}<c_{m+1,m}^{(x,z,t,\theta)}.
$$
Note now that
\begin{equation}\label{eq:compare_c_seqs_first_terms}
\frac{c_{k+1,m}^{(x,z,t,\theta)}}{c_{k,m}^{(x,z,t,\theta)}}=\frac{b^{(t,\theta)}_{k+1}(m)}{b^{(t,\theta)}_{k}(m)}=h_m(k)e^{(2k+|\theta|)t/2}\leq (|\theta|+2k+1)e^{(2k+|\theta|)t/2},
\end{equation}
where $h_m(k):=\frac{|\theta|+m+k-1}{k-m+1}\frac{|\theta|+2k+1}{|\theta|+2k-1}$, the second equality follows from (\ref{eq:b_coefs}), and the last inequality holds because $h_m(k)<h_k(k)=|\theta|+2k+1$, see the proof of Proposition 1 in \cite{Jenkins2017}.

Because by hypothesis $m\geq D_{\epsilon}^{(t,\theta)}$, recalling the definition of 
$D_{\epsilon}^{(t,\theta)}$ in (\ref{eq:D_constant}) and choosing $k=m+1$ in (\ref{eq:compare_c_seqs_first_terms}) yields
\begin{equation}\label{eq:compare_c_seqs_first_terms2}
c_{m+2,m}< (|\theta|+2k+1)e^{(2k+|\theta|)t/2}c_{m+1,m}^{(x,z,t,\theta)}<(1-\epsilon)c_{m+1,m}^{(x,z,t,\theta)}.
\end{equation}
Finally, 
\begin{align*}
\frac{c_{m+1,m+1}^{(x,z,t,\theta)}}{c_{m+1,m}^{(x,z,t,\theta)}}&=\frac{b^{(t,\theta)}_{m+1}(m+1)}{b^{(t,\theta)}_{m+1}(m)}\frac{\text{E}[\mathcal{D}_{\theta+L_{m+1}}(z)]}{\text{E}[\mathcal{D}_{\theta+L_{m}}(z)]}
=\frac{|\theta|+2m}{(m+1)(|\theta|+m)}
\frac{\text{E}[\mathcal{D}_{\theta+L_{m+1}}(z)]}{\text{E}[\mathcal{D}_{\theta+L_{m}}(z)]}\\
&=\frac{1}{(m+1)}\left(1+\frac{m}{|\theta|+m}\right)
\frac{\text{E}[\mathcal{D}_{\theta+L_{m+1}}(z)]}{\text{E}[\mathcal{D}_{\theta+L_{m}}(z)]}
<\frac{2}{(m+1)}
\frac{\text{E}[\mathcal{D}_{\theta+L_{m+1}}(z)]}{\text{E}[\mathcal{D}_{\theta+L_{m}}(z)]}
<\epsilon,
\end{align*}
where the last inequality follows because $m+1>m\geq 2\tilde{K}^{(\theta,x,z)}/\epsilon$ and using Lemma \ref{lem:decreas_expec},
yielding
$$
c_{m+1,m+1}^{(x,z,t,\theta)}+c_{m+2,m}<\epsilon c_{m+1,m}^{(x,z,t,\theta)}+(1-\epsilon)c_{m+1,m}^{(x,z,t,\theta)}=c_{m+1,m}^{(x,z,t,\theta)},
$$
which concludes the proof.
\end{proof}

\bibliographystyle{plain}
\bibliography{ms}

\begin{thebibliography}{10}

\bibitem{Aurell2019a}
Erik Aurell, Magnus Ekeberg, and Timo Koski.
\newblock On a multilocus {W}right-{F}isher model with mutation and a
  {S}virezhev-{S}hahshahani gradient-like selection dynamics.
\newblock 2019.

\bibitem{Beaumont2002}
Mark~A. Beaumont, Wenyang Zhang, and David~J. Balding.
\newblock Approximate {B}ayesian computation in population genetics.
\newblock {\em Genetics}, 162(4):2025--2035, 2002.

\bibitem{Beskos2006}
Alexandros Beskos, Omiros Papaspiliopoulos, and Gareth~O. Roberts.
\newblock Retrospective exact simulation of diffusion sample paths with
  applications.
\newblock {\em Bernoulli}, 12(6):1077--1098, 2006.

\bibitem{Beskos2008}
Alexandros Beskos, Omiros Papaspiliopoulos, and Gareth~O. Roberts.
\newblock A factorisation of diffusion measure and finite sample path
  constructions.
\newblock {\em Methodology and Computing in Applied Probability},
  10(1):85--104, Mar 2008.

\bibitem{Beskos2005}
Alexandros Beskos and Gareth~O. Roberts.
\newblock Exact simulation of diffusions.
\newblock {\em Annals of Applied Probability}, 15(4):2422--2444, 2005.

\bibitem{Blanchet2017}
Jose Blanchet and Fan Zhang.
\newblock Exact simulation for multivariate {I}t\^{o} diffusions.
\newblock {\em arXiv preprint arXiv:1706.05124}, 2017.

\bibitem{Burger2000}
Reinhard Bürger.
\newblock {\em The Mathematical Theory of Selection, Recombination, and
  Mutation}.
\newblock Wiley, 2000.

\bibitem{Casella2008}
Bruno Casella and Gareth~O. Roberts.
\newblock Exact {M}onte {C}arlo simulation of killed diffusions.
\newblock {\em Advances in Applied Probability}, 40(1):273--291, 2008.

\bibitem{Chen2013}
Nan Chen and Zhengyu Huang.
\newblock Localization and exact simulation of {B}rownian motion-driven
  stochastic differential equations.
\newblock {\em Mathematics of Operations Research}, 38(3):591--616, 2013.

\bibitem{Corander2017}
Jukka Corander, Christophe Fraser, Michael~U. Gutmann, Brian Arnold, William~P.
  Hanage, Stephen~D. Bentley, Marc Lipsitch, and Nicholas~J. Croucher.
\newblock Frequency-dependent selection in vaccine-associated pneumococcal
  population dynamics.
\newblock {\em Nature Ecology \& Evolution}, 1(12):1950--1960, December 2017.

\bibitem{Crow1970}
James~F. Crow and Motoo Kimura.
\newblock {\em An introduction to population genetics theory.}
\newblock New York, Evanston and London: Harper \& Row, Publishers, 1970.

\bibitem{Devroye2006}
Luc Devroye.
\newblock Nonuniform random variate generation.
\newblock {\em Handbooks in operations research and management science},
  13:83--121, 2006.

\bibitem{Ekeberg2013}
Magnus Ekeberg, Cecilia L\"ovkvist, Yueheng Lan, Martin Weigt, and Erik Aurell.
\newblock Improved contact prediction in proteins: Using pseudolikelihoods to
  infer {P}otts models.
\newblock {\em Phys. Rev. E}, 87:012707, Jan 2013.

\bibitem{Etheridge2009}
Alison~M. Etheridge and Robert~C. Griffiths.
\newblock A coalescent dual process in a {M}oran model with genic selection.
\newblock {\em Theoretical Population Biology}, 75(4):320 -- 330, 2009.

\bibitem{Ethier2009}
Stewart~N. Ethier and Thomas~G. Kurtz.
\newblock {\em Markov processes: characterization and convergence}.
\newblock John Wiley \& Sons, 2009.

\bibitem{ethier1989}
Stewart~N. Ethier and Thomas Nagylaki.
\newblock Diffusion approximations of the two-locus {W}right-{F}isher model.
\newblock {\em Journal of Mathematical Biology}, 27(1):17--28, Feb 1989.

\bibitem{Ewens2004}
Warren~J. Ewens.
\newblock {\em Mathematical population genetics}.
\newblock Interdisciplinary applied mathematics. Springer, New York, 2nd
  edition, 2004.

\bibitem{Favero2020}
Martina Favero, Henrik Hult, and Timo Koski.
\newblock A dual process for the coupled wright-fisher diffusion.
\newblock {\em arXiv:1906.02668}, page (Under revision for J. Theor. Biology),
  2019.

\bibitem{Fearnhead2006}
Paul Fearnhead.
\newblock The stationary distribution of allele frequencies when selection acts
  at unlinked loci.
\newblock {\em Theoretical Population Biology}, 70(3):376--386, 2006.

\bibitem{Fearnhead2012}
Paul Fearnhead and Dennis Prangle.
\newblock Constructing summary statistics for approximate {B}ayesian
  computation: semi-automatic approximate {B}ayesian computation.
\newblock {\em Journal of the Royal Statistical Society: Series B (Statistical
  Methodology)}, 74(3):419--474, 2012.

\bibitem{Fitzsimmons1993}
Pat Fitzsimmons, Jim Pitman, and Marc Yor.
\newblock {\em Seminar on stochastic processes, progress in probability},
  chapter Markovian Bridges: Construction, Palm Interpretation, and Splicing,
  pages 101--134.
\newblock Number~33. Birkh{\"a}user Boston, Boston, MA, 1993.

\bibitem{Gao2018}
Chen-Yi Gao, Hai-Jun Zhou, and Erik Aurell.
\newblock Correlation-compressed direct-coupling analysis.
\newblock {\em Phys. Rev. E}, 98:032407, Sep 2018.

\bibitem{Griffiths1980}
Robert~C. Griffiths.
\newblock Lines of descent in the diffusion approximation of neutral
  wright-fisher models.
\newblock {\em Theoretical Population Biology}, 17(1):37 -- 50, 1980.

\bibitem{Griffiths1984}
Robert~C. Griffiths.
\newblock Asymptotic line-of-descent distributions.
\newblock {\em Journal of Mathematical Biology}, 21(1):67--75, 1984.

\bibitem{Griffiths2006}
Robert~C. Griffiths.
\newblock Coalescent lineage distributions.
\newblock {\em Advances in Applied Probability}, 38(2):405--429, 2006.

\bibitem{Griffiths2018}
Robert~C. Griffiths, Paul~A. Jenkins, and Dario Spanò.
\newblock Wright–fisher diffusion bridges.
\newblock {\em Theoretical Population Biology}, 122:67 -- 77, 2018.

\bibitem{Griffiths1983}
Robert~C. Griffiths and Wen-Hsiung Li.
\newblock Simulating allele frequencies in a population and the genetic
  differentiation of populations under mutation pressure.
\newblock {\em Theoretical Population Biology}, 23(1):19 -- 33, 1983.

\bibitem{Griffiths2010}
Robert~C. Griffiths and Dario Span\`{o}.
\newblock Diffusion processes and coalescent trees.
\newblock In {\em Probability and Mathematical Genetics, Papers in Honour of
  Sir John Kingman (N.H. Bingham and C.M. Goldie eds.)}, volume 378 of {\em LMS
  Lecture Note Series}, chapter~15, pages 358--375. Cambridge University Press,
  2010.

\bibitem{Hoessjer2016}
Ola Hössjer, Peder~A. Tyvand, and Touvia Miloh.
\newblock Exact markov chain and approximate diffusion solution for haploid
  genetic drift with one-way mutation.
\newblock {\em Mathematical biosciences}, 272:100--12, Feb 2016.

\bibitem{Jenkins2013}
Paul~A Jenkins.
\newblock Exact simulation of the sample paths of a diffusion with a finite
  entrance boundary.
\newblock {\em arXiv preprint arXiv:1311.5777}, 2013.

\bibitem{Jenkins2017}
Paul~A. Jenkins and Dario Span\`{o}.
\newblock Exact simulation of the {W}right–-{F}isher diffusion.
\newblock {\em Annals of Applied Probability}, 27(3):1478--1509, 2017.

\bibitem{Jewett2014}
Ethan~M. Jewett and Noah~A. Rosenberg.
\newblock Theory and applications of a deterministic approximation to the
  coalescent model.
\newblock {\em Theoretical Population Biology}, 93:14 -- 29, 2014.

\bibitem{Karatzas1998}
Ioannis Karatzas and Steven~E. Shreve.
\newblock {\em Brownian Motion and Stochastic Calculus}, volume 113 of {\em
  Graduate Texts in Mathematics}.
\newblock Springer New York, NY, 2nd edition, 1998.

\bibitem{Kimur1955}
Motoo Kimura.
\newblock Stochastic processes and distribution of gene frequencies under
  natural selection.
\newblock {\em Cold Spring Harbor Symposia on Quantitative Biology},
  (20):33--53, 1955.

\bibitem{Kimura1964}
Motoo Kimura.
\newblock Diffusion models in population genetics.
\newblock {\em Journal of Applied Probability}, 1(2):177–232, 1964.

\bibitem{Kloeden1992}
Peter~E. Kloeden and Eckhard Platen.
\newblock Numerical solution of stochastic differential equations.
\newblock 1992.

\bibitem{Shraiman2011}
Richard~A. Neher and Boris~I. Shraiman.
\newblock {Statistical genetics and evolution of quantitative traits}.
\newblock {\em Rev. Mod. Phys.}, 83:1283--1300, Nov 2011.

\bibitem{Nene2018}
Nuno~R. Nené, Ville Mustonen, and Christopher J.~R. Illingworth.
\newblock Evaluating genetic drift in time-series evolutionary analysis.
\newblock {\em Journal of Theoretical Biology}, 437:51 -- 57, 2018.

\bibitem{Pollock2016}
Murray Pollock, Adam~M. Johansen, and Gareth~O. Roberts.
\newblock On the exact and $\varepsilon$-strong simulation of (jump)
  diffusions.
\newblock {\em Bernoulli}, 22(2):794--856, 2016.

\bibitem{Schubert2019}
Benjamin Schubert, Rohan Maddamsetti, Jackson Nyman, Maha~R. Farhat, and
  Debora~S. Marks.
\newblock Genome-wide discovery of epistatic loci affecting antibiotic
  resistance in {N}eisseria gonorrhoeae using evolutionary couplings.
\newblock {\em Nature Microbiology}, 4(2):328--338, February 2019.

\bibitem{Skwark2017}
Marcin~J. Skwark, Nicholas~J. Croucher, Santeri Puranen, Claire Chewapreecha,
  Maiju Pesonen, Ying~Ying Xu, Paul Turner, Simon~R. Harris, Stephen~B. Beres,
  James~M. Musser, et~al.
\newblock Interacting networks of resistance, virulence and core machinery
  genes identified by genome-wide epistasis analysis.
\newblock {\em PLoS genetics}, 13(2):e1006508, 2017.

\bibitem{Steinrucken2014}
Matthias Steinrücken, Anand Bhaskar, and Yun~S. Song.
\newblock A novel spectral method for inferring general diploid selection from
  time series genetic data.
\newblock {\em Annals of Applied Statistics}, 8(4):2203--2222, 2014.

\bibitem{Tataru2017}
Paula Tataru, Maria Simonsen, Thomas Bataillon, and Asger Hobolth.
\newblock Statistical inference in the {W}right–{F}isher model using allele
  frequency data.
\newblock {\em Systematic Biology}, 66(1):e30--e46, 2017.

\bibitem{Tavare1997}
S.~Tavaré, D.~J. Balding, R.~C. Griffiths, and P.~Donnelly.
\newblock Inferring coalescence times from {DNA} sequence data.
\newblock {\em Genetics}, 145(2):505--518, February 1997.

\end{thebibliography}

\end{document}